\documentclass[12pt,a4paper]{amsart}

\title{Poincar\'{e} profiles of groups and spaces}
\author{David Hume}
\thanks{The first and third authors were supported by the grant ANR-14-CE25-0004 ``GAMME''.  The first and second authors were supported in part by the NSF grant DMS-1440140 while in residence at the Mathematical Sciences Research Institute in Berkeley, California, during the Fall 2016 semester.  The second author was also supported in part by EPSRC grant EP/P010245/1.}
\address{Mathematical Institute, University of Oxford, Oxford, OX2 6GG.}
\email{david.hume@maths.ox.ac.uk}
\author{John M. Mackay}
\address{School of Mathematics, University of Bristol, Bristol, BS8 1TX.}
\email{john.mackay@bristol.ac.uk}
\author{Romain Tessera}
\address{Universit\'{e} Paris-Sud, Orsay, France.}
\email{romain.tessera@math.u-psud.fr}
\date{\today}

\usepackage{amsmath,amsthm,amsfonts,graphicx,amssymb}
\usepackage{hyperref}
\usepackage{xcolor}
\usepackage[shortlabels]{enumitem}

\newtheorem{btheorem}{Theorem}[]
\newtheorem{bcorollary}[btheorem]{Corollary}
\newtheorem{bproposition}[btheorem]{Proposition}

\numberwithin{equation}{section}
\newtheorem{theorem}[equation]{Theorem}
\newtheorem{proposition}[equation]{Proposition}
\newtheorem{corollary}[equation]{Corollary}
\newtheorem{lemma}[equation]{Lemma}

\theoremstyle{definition}

\newtheorem{example}[equation]{Example}

\newtheorem{bremark}[btheorem]{Remark}

\newtheorem{definition}[equation]{Definition}
\newtheorem{remark}[equation]{Remark}

\newtheorem*{theorem*}{Theorem}
\newtheorem*{assump*}{Standing assumption}
\newtheorem*{remark*}{Remark}
\newtheorem*{claim*}{Claim}

\newtheoremstyle{citing}
  {3pt}
  {3pt}
  {\itshape}
  {}
  {\bfseries}
  {}
  {.5em}
  {\thmnote{#3}}

\theoremstyle{citing}

\DeclareMathOperator{\diam}{diam}

\DeclareMathOperator{\cut}{cut}
\DeclareMathOperator{\sep}{sep}

\newcommand{\eps}{\epsilon}

\newcommand{\bdry}{\partial_\infty}

\newcommand{\hLip}{h_{\mathrm{Lip}}}
\DeclareMathOperator{\Lip}{Lip}
\DeclareMathOperator{\Isom}{Isom}
\DeclareMathOperator{\mdim}{mdim}

\newcommand{\set}[1]{\left\{#1\right\}}
\newcommand{\setcon}[2]{\left\{#1\, :\, #2\right\}}
\newcommand{\varsetcon}[2]{\left\{ #1\, :\, #2\right\}}

\newcommand{\abs}[1]{\left\lvert#1\right\rvert}
\newcommand{\norm}[1]{\left\lvert\left\lvert#1\right\rvert\right\rvert}

\newcommand{\cH}{\mathcal{H}}

\newcommand{\ra}{\rightarrow}
\newcommand{\R}{\mathbb{R}}
\newcommand{\C}{\mathbb{C}}
\newcommand{\Sph}{\mathbb{S}}
\newcommand{\N}{\mathbb{N}}
\newcommand{\Z}{\mathbb{Z}}

\newcommand{\HH}{\mathbb{H}}

\def\Xint#1{\mathchoice
{\XXint\displaystyle\textstyle{#1}}%
{\XXint\textstyle\scriptstyle{#1}}%
{\XXint\scriptstyle\scriptscriptstyle{#1}}%
{\XXint\scriptscriptstyle\scriptscriptstyle{#1}}%
\!\int}
\def\XXint#1#2#3{{\setbox0=\hbox{$#1{#2#3}{\int}$}
\vcenter{\hbox{$#2#3$}}\kern-.5\wd0}}

\def\dashint{\Xint-}

\numberwithin{equation}{section}

\begin{document}

\begin{abstract} We introduce a spectrum of monotone coarse invariants for metric measure spaces called Poincar\'{e} profiles. The two extremes of this spectrum determine the growth of the space, and the separation profile as defined by Benjamini--Schramm--Tim\'{a}r. In this paper we focus on properties of the Poincar\'{e} profiles of groups with polynomial growth, and of hyperbolic spaces, where we deduce a  connection between these profiles and conformal dimension.  As applications, we use these invariants to show the non-existence of coarse embeddings in a variety
of examples.
\end{abstract}

\maketitle

\setcounter{tocdepth}{1}
\tableofcontents

\section{Introduction}
One of the key goals of geometric group theory is to relate algebraic and large-scale geometric properties of finitely generated groups. Each finitely generated group is considered as a metric space when equipped with a word metric; different finite generating sets yield different metrics but they are all \textbf{quasi-isometric}.

Interesting families of groups are generally either defined algebraically and are geometrically mysterious (such as solvable groups) or are defined geometrically and are algebraically mysterious (such as hyperbolic groups). 

In general, subgroups of finitely generated groups need not be finitely generated. Even when they are, the word metric on the subgroup is not always quasi-isometric to the word metric of the group restricted to the subgroup. However, not all geometric information is lost as the inclusion of a subgroup is always a coarse embedding:

A map $\phi:X\to Y$ between metric spaces is a coarse embedding if there exist increasing functions $\rho_{\pm}:[0,\infty)\to [0,\infty)$ such that $\rho_-(r)\to\infty$ as $r\to\infty$ and for all $x,y\in X$
\[
 \rho_-(d_X(x,y)) \leq d_Y(\phi(x),\phi(y)) \leq \rho_+(d_X(x,y)).
\]

\subsection*{The geometric subgroup problem}
Let $G$ be a finitely generated group. For which finitely generated groups $H$ is there a coarse embedding of $H$ into $G$?

Note that this is a generalisation of the problem of determining the finitely generated subgroups of $G$. For example, there are hyperbolic groups into which $\Z^2$ can be coarsely embedded, but $\Z^2$ can never occur as a subgroup of a hyperbolic group.

To make progress with the geometric subgroup problem requires \textbf{monotone coarse invariants}: maps $M$ from finitely generated groups to a poset $(P,\leq)$ such that if $H$ coarsely embeds into $G$ then $M(H)\leq M(G)$. Growth and asymptotic dimension are two monotone coarse invariants which have been intensively studied in the literature; much more recently Benjamini--Schramm--Tim\'{a}r introduced a third\footnote{To our knowledge these are the only three monotone coarse invariants for all finitely generated groups which may take infinitely many different values.} monotone coarse invariant called the \textbf{separation profile}. The separation profile of an infinite, bounded degree graph at $r\in \N$ is the
supremum over all subgraphs of size 
 $\leq r$, of the number of vertices needed to be removed from the subgraph, in order to cut it to connected pieces of size at most $r/2$.

\medskip

In this paper, we introduce for every $1\leq p\leq \infty$, an analytic generalization of the separation profile that we call the $L^p$-Poincar\'e profile. It is indeed a generalization as for $p=1$ we actually recover the separation profile; in a sense these invariants interpolate between separation ($p=1$) and growth ($p=\infty$, as we see later). 
We shall establish general properties of these invariants, calculate their values in natural examples like groups of polynomial growth and certain hyperbolic groups, and use these to find coarse non-embedding results.

\subsection{Poincar\'e profiles}

For $p\in[1,\infty]$, the $L^p$--Poincar\'{e} constant of a finite graph $\Gamma$ is
\[
 h^p(\Gamma) = \inf\setcon{\frac{\norm{\nabla f}_p}{\norm{f}_p}}{f:V\Gamma\to\R\ \sum_{v\in V\Gamma} f(v)=0}
\]
where $\nabla f:V\Gamma\to\R$ is defined by \[\nabla f(x)=\max\setcon{\abs{f(x)-f(y)}}{xy\in E\Gamma}.\] We then define the $L^p$--Poincar\'{e} profile of a bounded degree graph $X$ by
\[
 \Lambda^p_X(r) = \sup\setcon{\abs{V\Gamma}h^p(\Gamma)}{\Gamma\leq X,\ \abs{V\Gamma}\leq r}.
\]
We consider functions up to the natural order $\lesssim$ where $f\lesssim g$ if there exists a constant $C$ such that $f(r)\leq Cg(Cr+C) + C$ for all $r$, and $f\simeq g$ if $f\lesssim g$ and $g\lesssim f$. Often, the constant $C$ will depend on some parameter $p$; to emphasise this we will use the notations $\lesssim_p$ and $\simeq_p$.

In this introduction, we only define Poincar\'{e} profiles in the context of graphs; however, our results naturally extend to compactly generated locally compact groups and Riemannian manifolds with bounded geometry. The majority of the paper is presented in a more general context which includes all of these spaces.
\medskip

A lower bound on the $L^p$-Poincar\'{e} profile corresponds to a ``$p$-Poin\-ca\-r\'e inequality''\footnote{Technically these Poincar\'e inequalities are Neumann-type, considering functions with average $0$, rather than Dirichlet-type Poincar\'{e} inequalities which consider only functions which are $0$ on the boundary of the subgraph in the ambient space. See Remark~\ref{rem:NeumannDirich} for more details.} for functions on a finite subgraph of the corresponding size. Poincar\'e inequalities have been intensively studied, particularly in the case of balls in doubling metric spaces, see \cite{Saloff-Coste,HajKos00-Sobolev-met-Poincare} and references therein.
For finite graphs, there is a vast literature linking Cheeger constants and spectral gaps to such inequalities when $p=1,2$, see~\cite{Chung-Fan-97-Spectral-graph-theory,Saloff-Coste-96-St-Flour-notes}.
Discrete Poincar\'e inequalities on balls in metric spaces have been studied before by, for example, Holopainen--Soardi~\cite{Holo-Soa-97-p-harmonic} and Gill--Lopez~\cite{Gill-Lopez-15-discrete-PI}.  Our approach differs in that we are working in a situation where global Poincar\'e inequalities do not necessarily hold, where measures need not be doubling, and where we have to consider inequalities on all subsets, not just balls.

One of our main results is that like the separation profile, Poincar\'{e} profiles are monotonous under regular maps. 

\begin{btheorem}\label{bthm:monotoneunderreg} Let $X,Y$ be graphs with bounded degree. If there is a regular map $F:VX\to VY$, then for all $p\in[1,\infty]$, $\Lambda^p_X\lesssim_p \Lambda^p_Y$.
\end{btheorem}
A map $F:(X,d_X)\to (Y,d_Y)$ between metric spaces is said to be \textbf{regular} if there exists a constant $K$ such that $d_Y(F(x),F(y))\leq K(1+d_X(x,y))$ and for every $y\in Y$, $F^{-1}(B(y,1))$ is contained in a union of at most $K$ balls of radius $1$ in $X$. In particular, every quasi-isometric or coarse embedding of bounded degree graphs is regular. Thus for each $p$ the $L^p$-Poincar\'e profile is a well-defined coarse invariant of a finitely generated group $G$.

\subsection{Extremal cases}

In the cases $p=\infty$ and $p=1$ the Poincar\'{e} profile is easily understood in terms of the growth and separation profile respectively.

Recall the \textbf{growth function} of a graph $X$: $\gamma_X(k)$ is the maximum number of vertices contained in a closed ball $B(x,k)$ of radius $k$ centred at some vertex $x\in VX$. We define the \textbf{inverse growth function}: $\kappa_X(r)$ is the smallest positive $k$ such that $\gamma_X(k)>r$.

At one extreme, $p=\infty$, the Poincar\'{e} profile detects inverse growth.

\begin{bproposition}\label{bprop:P8growth} For any bounded degree graph $X$, $\Lambda^\infty_X(r) \simeq {\displaystyle \sup_{1\leq s \leq r}}\frac{\displaystyle s}{\displaystyle \kappa_X(s)}$.
\end{bproposition}

From this, we may easily deduce Theorem \ref{bthm:monotoneunderreg} in the case $p=\infty$. At the other extreme we show that the $L^1$-Poincar\'{e} profile is equivalent to the separation profile, as introduced by Benjamini--Schramm--Tim\'{a}r \cite{BenSchTim-12-separation-graphs}. The perspective we adopt of studying Poincar\'{e} profiles up to regular maps is inspired by their observation that separation is monotone under regular maps.

In  \cite{HumSepExp} the first author proved that the separation profile of an infinite graph $X$ may be defined by $\sep_X(r)=\max\set{\abs{\Gamma}h(\Gamma)}$ where the maximum is taken over all subgraphs $\Gamma$ of $X$ with at most $r$ vertices, and $h(\Gamma)$ is the Cheeger constant.  The following fact is then an easy consequence of the classical co-area formula (see section \ref{section:Cheeger/L1}).

\begin{bproposition}\label{bprop:P1sep} For any bounded degree graph $X$, $\Lambda^1_X(r) \simeq \sep_X(r)$.
\end{bproposition}

\begin{bremark} The case of $p=2$ is also natural, being the largest spectral
gap among subgraphs of a given size. The spectral gap can be used to bound mixing times of random walks on the subgraph. A related \emph{spectral profile} was considered by Goel--Montenegro--Tetali \cite{GMT_Mixing}.
\end{bremark}

\subsection{Relating profiles}

The following results are classical, and are likely to be easy exercises for experts; for completeness we present full proofs.

\begin{bproposition}\label{bprop:pqmonotone} Let $1\leq p \leq q<\infty$. There exists a constant $C=C(p,q)$ such that for every bounded degree graph $X$ and every $r$ we have $\Lambda^p_X(r) \leq C \Lambda^q_X(r)$.
\end{bproposition}

In the opposite direction we have the following.

\begin{bproposition}\label{bprop:p1monotone} If $\Gamma$ is a finite graph and $p\in[1,\infty)$, then $h^p(\Gamma)^p\leq 2^{p}h^1(\Gamma)$.
\end{bproposition}
Asymptotically this is sharp for balls in the $3$-regular tree, as we will see in section \ref{sec:trees}.  Proposition \ref{bprop:pqmonotone} cannot be extended to the case $q=\infty$ since there are bounded degree graphs containing expanders: combining the above propositions with results in \cite{HumSepExp} we see that for every $p\in[1,\infty)$, $\Lambda^p_X(r)/r\not\to 0$ as $r\to\infty$ if and only if $X$ contains an expander, while a bounded degree graph $Y$ has at most exponential growth, so always satisfies $\Lambda^\infty_Y(r)\lesssim r/\log(r)$.

\subsection{Polynomial growth}
 In the the case of abelian groups, Benjamini, Schramm and Timar show that for $d\geq 2$, $\sep_{\Z^d}(r)\simeq r^{\frac{d-1}{d}}$, therefore the separation profile detects the degree of volume growth \cite{BenSchTim-12-separation-graphs}. We extend their result to groups with polynomial growth. 
Recall that Gromov's celebrated polynomial growth theorem asserts that every finitely generated group with polynomial growth is virtually nilpotent. Results of Bass--Guivarc'h then show that for every group $G$ of polynomial growth there is an integer $d$ such that $\gamma_G(r)\simeq r^d$ \cite{Gro-81-poly-growth, Bass-72-degree-poly-growth, Gui-73-crois-poly}.

\begin{btheorem}\label{bthm:CGLC-poly} Let $G$ be a finitely generated group such that $\gamma_G(r)\simeq r^d$. Then for all $p\in[1,\infty]$, $\Lambda^p_G(r)\simeq_p r^{\frac{d-1}{d}}$.
\end{btheorem}
To prove the lower bound on $\Lambda^p_G(r)$ we calculate a lower bound on the separation profile using a well-known Poincar\'{e} inequality on balls in such groups, and apply Propositions \ref{bprop:P1sep} and \ref{bprop:pqmonotone}. For the upper bound we use a general result, Proposition~\ref{prop:mdim}, which holds for any bounded degree graph with finite Assouad--Nagata dimension (cf.\ \cite[Theorem $1.5$]{HumSepExp}).

Recall that by a classical result of Heintze \cite{Heintze}, every simply connected negatively curved homogeneous Riemannian manifold $M$ is isometric to a connected Lie group of the form $N\rtimes \R$ equipped with a left-invariant Riemannian metric, where $N$ is a simply connected nilpotent Lie group and the action of $\R$ on $N$ is contracting. We immediately deduce from  Theorem \ref{bthm:CGLC-poly} that for every $p\in[1,\infty]$, the $L^p$-Poincar\'e profile of such a manifold\footnote{In our more general context we can directly consider Poincar\'e profiles of manifolds. For the purposes of the introduction one can consider any bounded degree graph quasi-isometric to the manifold.} is bounded from below by $r^{\frac{d-1}{d}}$, where $d$ is the homogeneous dimension of $N$. As a special case of this we deduce the following lower bounds for rank one symmetric spaces.

\begin{bcorollary}\label{bcor:lbdrank1ss}
	Let $\mathbb{K} \in \{ \R, \C, \HH, \mathbb{O}\}$ be a real division algebra, and let $X=\HH_{\mathbb{K}}^m$ be a rank-one symmetric space for $m \geq 2$ (and $m=2$ when $\mathbb{K}=\mathbb{O}$).
	Then, for all $1\leq p<\infty$, we have $\Lambda_X^p(r) \gtrsim_p r^{(Q-1)/Q}$ where $Q = (m+1)\dim_{\R} \mathbb{K} -2$.
\end{bcorollary}

It is worth noting at this point that $Q$ is the conformal dimension of the boundary of $X$. For large $p$ this bound is far from optimal as we will see in the next section.

\subsection{Hyperbolic spaces}
For groups quasi-isometric to real hyperbolic spaces  Benjamini, Schramm and Timar observed a strange phenomenon: while for $d\geq 2$, $\sep_{\HH^{d+1}}(r)\simeq r^{\frac{d-1}{d}}$, one has $\sep_{\HH^{2}}(r)\simeq \log r$. In the case $d\geq 2$, graphs contained in horospheres (isomorphic to $\R^d$) are the hardest to cut, while for $d=1$, balls are the best connected subsets (note that in this case horospheres are real lines).  As we shall see below,  this ``anomaly" at $d=1$ becomes part of a more natural phenomenon in the context of Poincar\'e profiles. 

We begin by considering the case of an infinite $3$-regular tree.

\begin{btheorem}\label{bthm:trees} Let $T$ be the infinite $3$-regular tree. Then $\Lambda^p_T(r)\simeq_p r^{\frac{p-1}{p}}$, for all $p\in[1,\infty)$.
\end{btheorem}

Note that when $p=\infty$, $\Lambda^p_T(r)\simeq r/\log(r)$ by Proposition \ref{prop:Linfty_inverse growth}. Using Theorem \ref{bthm:trees}, together with results of Chou and Benjamini--Schramm on embeddings of trees into elementary amenable groups with exponential growth and non-amenable groups respectively \cite{Chou, Ben-Sch-trees-in-nonamen} we obtain the following corollary.

\begin{bcorollary}\label{bcor:lowerbdsolvnonamen} Let $G$ be a finitely generated elementary amenable group with exponential growth or a finitely generated infinite non-amenable group. Then for all $p\in[1,\infty)$, $\Lambda^p_G(r)\gtrsim_p r^{\frac{p-1}{p}}$.
\end{bcorollary}

The following result suggests a connection between Pansu's conformal dimension of the boundary of a hyperbolic group (see Definition~\ref{def:confdim}), and a phase transition in the Poincar\'{e} profiles of hyperbolic groups.   

\begin{btheorem}\label{bthm:hypupperbd} Let $G$ be a finitely generated hyperbolic group with equivariant conformal dimension $Q$ (see Definition~\ref{def:equiv-confdim}). For every $\varepsilon>0$
\[
 \Lambda^p_G(r) \lesssim \left\{
 \begin{array}{lcl}
 r^{\frac{Q-1}{Q}+\epsilon} & \textup{if} & p \leq Q
 \\
 r^{\frac{p-1}{p}} & \textup{if} & p>Q.
 \end{array}\right.
\]
If the equivariant conformal dimension is attained, we have:
\[
 \Lambda^p_G(r) \lesssim \left\{
 \begin{array}{lcl}
 r^{\frac{Q-1}{Q}} & \textup{if} & 1\leq p < Q
 \\
 r^{\frac{Q-1}{Q}} \log^{\frac{1}{Q}}(r) & \textup{if} & p = Q
 \\
 r^{\frac{p-1}{p}} & \textup{if} & p>Q.
 \end{array}\right.
\]
\end{btheorem}

In certain cases we are able to prove that these upper bounds are sharp. Our key examples are rank-one symmetric spaces, and a collection of groups $G_{m,n}=\langle s_1,\ldots,s_m\mid s_i^n,[s_1,s_{2}],\ldots,[s_{m-1},s_m],[s_m,s_1] \rangle$, $m\geq 5,n\geq 3$ which occur as uniform lattices in the isometry groups of associated Fuchsian buildings $\Delta_{m,n}$, as studied by Bourdon and Bourdon--Pajot amongst others~\cite{Bou-97-GAFA-exact-cdim,BP-99-hyp-build-PI}.  Following the terminology of~\cite{Cap-14-aut-ra-buildings} we call these \emph{Bourdon buildings}.  These groups are virtually torsion free~\cite{HsuWise}, and commensurable to hyperbolic Coxeter groups when $n$ is even~\cite{Hag-commens}.

\begin{btheorem}\label{bthm:hypconfdim}
	Let $X=\HH_{\mathbb{K}}^m$ be a rank-one symmetric space for $\mathbb{K} \in \{ \R, \C, \HH, \mathbb{O}\}$ and $m \geq 2$ (with $m=2$ when $\mathbb{K}=\mathbb{O}$),
	or let $X$ be one of the groups $G_{m,n}$; in either case, let $Q$ be the conformal dimension of the boundary of $X$. Then
\[
\Lambda^p_{X}(r) \simeq \left\{
\begin{array}{ll}
r^{\frac{Q-1}{Q}}
&
\textrm{if } p< Q
\\
r^{\frac{Q-1}{Q}}\log(r)^{\frac{1}{Q}}
&
\textrm{if } p = Q
\\
r^{\frac{p-1}{p}}
&
\textrm{if } p > Q
\end{array}\right.
\]
\end{btheorem}
It is interesting to note that uniform lattices $G$ in $PSL(2,\R)$ satisfy $\Lambda^1_G(r)\simeq\log(r)$ and $\Lambda^p_G(r)\simeq r^{\frac{p-1}{p}}$ for all $p>1$, while non-uniform lattices $H$ satisfy $\Lambda^1_H(r)\simeq r^{\frac{p-1}{p}}$ for all $p\geq 1$. We have no other examples of this distinction between uniform and non-uniform lattices for any other $p$ or for any groups of higher rank.

Conformal dimension is defined in Definition~\ref{def:confdim}, but we note here that for a rank-one symmetric space $\HH^m_\mathbb{K}$ we have $Q=(m+1)\dim_\R \mathbb{K}-2$ and for the groups $G_{m,n}$ we have $Q=1+\log(n-1)/\mathrm{arccosh}((m-2)/m)$, which takes a dense set of values in $(1,\infty)$ as $m,n$ vary.
 
The upper bound in Theorem \ref{bthm:hypconfdim} is obtained by constructing specific functions on the boundary using an embedding of the space into a real hyperbolic space.  The lower bound in Theorem \ref{bthm:hypconfdim} for rank-one symmetric spaces with $p<Q$ follows from Corollary~\ref{bcor:lbdrank1ss}.
For the groups $G_{m,n}$, and for the sharp case $p=Q$, we require the following more general result.

\begin{btheorem}\label{bthm:lbdQrgPI}
	Suppose that $X$ is a visual Gromov hyperbolic graph with a visual metric $\rho$ on $\bdry X$ that is Ahlfors $Q$-regular and admits a $p$-Poincar\'e inequality.  Then for all $q\geq p$, $\Lambda_X^q(r)\gtrsim r^{(Q-1)/Q}$. 
   
	Moreover, if $(\bdry X, \rho)$ admits a $Q$-Poincar\'e inequality, then $\Lambda_X^Q(r)\gtrsim r^{1-1/Q}\log(r)^{1/Q}$. 
\end{btheorem}
Here a ``$p$-Poincar\'e inequality'' is in the sense of Heinonen--Koskela \cite{HK-98-qc-pi}, namely an analytic property of the compact metric space $\bdry X$.
Such inequalities hold on boundaries of rank-one symmetric spaces, see e.g.\ \cite{Jerison86-Poincare, HK-98-qc-pi, Mac-Tys-cdimexpo}, so we can apply this lower bound to obtain an alternative proof of Corollary \ref{bcor:lbdrank1ss}.  For the groups $G_{m,n}$, the Poincar\'{e} inequalities are constructed in \cite{BP-99-hyp-build-PI}.

The sharp lower bounds on $\Lambda^{Q}_X$ come from showing a suitable Poincar\'e inequality on annuli $B(o,2r)\setminus B(o,r)$ in $X$.  It is interesting to observe that for $p<Q$, $p=Q$, and $p>Q$, the sharp lower bounds on $\Lambda^{p} X$ are realised by embedded spheres, annuli and trees respectively.

Finally, Theorems \ref{bthm:trees} and \ref{bthm:hypconfdim}, together with the embedding theorem of Bonk--Schramm \cite{BS-00-gro-hyp-embed}, imply that for every hyperbolic group $G$ there is some $p_0$ such that for all $p>p_0$, we have $\Lambda^p_G(r)\simeq r^{\frac{p-1}{p}}$. The relationship between the infimal such $p_0$ and the conformal dimension of the boundary of $G$ is one of the most intriguing aspects of these profiles.

\subsection{Consequences}

By applying Theorem \ref{bthm:hypconfdim} to the groups $G_{m,n}$, we find a new collection of functions which can be obtained as separation profiles of finitely generated groups:
\begin{bcorollary}\label{bthm:denseexponents} There exists a dense subset $A$ of $(0,1)$ such that for all $\alpha\in A$ there is a hyperbolic group $G_\alpha$ with $\sep_{G_\alpha}(r)\simeq r^\alpha$.
\end{bcorollary}

The key purpose of a monotone coarse invariant is to be able to distinguish situations in which one space cannot be coarsely embedded into another.  Unlike the case of quasi-isometric embeddings, there are few general tools to do this; asymptotic dimension is one and growth (or equivalently, the $L^\infty$-Poincar\'{e} profile) is another. Here we present and discuss some results of this form which cannot be obtained by studying growth and/or asymptotic dimension.

\begin{bcorollary}\label{bcor:realcomphypspace} If there is a coarse embedding of $\mathbb{H}_\mathbb{C}^k$ into $\mathbb{H}_\mathbb{R}^l$, then $l>2k$.
  Likewise, if there is a coarse embedding of $\HH^k_\HH$ into $\HH^l_\R$, then $l>4k+2$.
\end{bcorollary}
To prove the analogous result for quasi-isometric embeddings, one can use the conformal dimension of the boundary, however, a coarse embedding does not necessarily induce a well-defined map between boundaries \cite{Baker-Riley} so this approach cannot be expected to work. Using asymptotic dimension as an invariant one could only deduce that $l\geq 2k$ in the first case and $l\geq 4k$ in the second.

By \cite{BS-00-gro-hyp-embed}, every hyperbolic group quasi-isometrically embeds into some $\mathbb{H}_\mathbb{R}^k$. A natural obstruction to a coarse embedding $G_k\to \mathbb{H}_\mathbb{R}^k$ is that the asymptotic dimension of $G_k$ is greater than $k$. Poincar\'{e} profiles provide a different obstruction.

\begin{bcorollary}\label{bcor:hyptohypk} For every $k$ there is a hyperbolic group $G_k$ of asymptotic dimension $2$ which does not coarsely embed into $\mathbb{H}_\mathbb{R}^k$.
\end{bcorollary}
We can take $G_k$ to be $G_{m(k),5}$ for some appropriately chosen $m(k)$ and apply Theorem $\ref{bthm:hypconfdim}$.

It is in general very difficult to prove a statement of the form ``a hyperbolic group $H$ is not isomorphic to a subgroup of a hyperbolic group $G$''. Two commonly considered obstructions are torsion and asymptotic dimension. Here we show that the Poincar\'{e} profiles can exclude subgroups when the two methods listed above fail.

\begin{bcorollary}\label{bcor:distinguishhypgps} There exists a collection of (torsion-free) hyperbolic groups $(G_q)_{q\in\mathbb{Q}}$ with asymptotic dimension $2$ such that whenever $i<j$ there is no coarse embedding from $G_i$ to $G_j$. In particular, $G_i$ is not virtually a subgroup of $G_j$.
\end{bcorollary}
Indeed, the groups $G_{m,n}$ are virtually torsion-free so we may choose the $G_q$ in Corollary \ref{bcor:distinguishhypgps} to be torsion-free.
By results of Gersten \cite{Gerst-hypdim2}, finitely presented subgroups of hyperbolic groups with cohomological dimension $2$ (which equals the asymptotic dimension for torsion-free hyperbolic groups \cite{BestMess,BuyaloLebedeva}) are hyperbolic but not necessarily quasi-convex. As a function of $q$, the conformal dimension of the boundary of $G_q$ will be strictly decreasing, therefore it will follow immediately that $G_i$ can never be a quasi-convex subgroup of $G_j$ for $i<j$.

\begin{bremark}
 By a recent result of Pansu \cite{Pan-16-coarse-conformal}, if a hyperbolic group $H$ coarsely embeds into a hyperbolic group $G$, then the ``$L^p$-cohomological dimension'' of $H$ is less than or equal to the conformal dimension of the boundary of $G$. In the cases of the Bourdon buildings and rank-one symmetric spaces, these two numbers turn out to coincide. This provides an alternative proof of 
 Corollaries \ref{bcor:realcomphypspace}, \ref{bcor:hyptohypk}  and \ref{bcor:distinguishhypgps}.  
\end{bremark}

\subsection{About the proofs}

The proof of the theorems described in the previous section employ a variety of techniques. In particular, the arguments needed for getting upper bounds are completely different than those for obtaining lower bounds. 

\subsubsection{Upper bounds.}
For groups with polynomial growth, our sharp upper bounds are obtained via an argument adapted from \cite{HumSepExp} based on the fact that these groups have finite Assouad--Nagata dimension (\cite[Theorem $1.5$]{HumSepExp} deals with the separation profile, corresponding to $p=1$). Finite Assouad--Nagata dimension is a quantitative strengthening of  finite asymptotic dimension. Contrary to the latter, 
finite Assouad--Nagata dimension is not monotone under coarse embedding (only quasi-isometric embedding) as it is sensitive to distortion of the metric. We come up with a new notion called {\it finite measured  dimension} (see Definition \ref{def:measurabledim}), weaker than finite asymptotic dimension, which should be of independent interest. Our motivation here is that it is well adapted for providing upper bound on the Poincar\'e profiles.

Obtaining upper bounds for hyperbolic groups is trickier, and occupies all of \S \ref{section:upperHyp}. We need three different arguments, depending whether $p$ lies below, above, or equals the (equivariant) conformal dimension. 
We show that hyperbolic groups admit in some sense ``many hyperplanes'', by using a theorem of Bonk--Schramm~\cite{BS-00-gro-hyp-embed} to embed the group into a real hyperbolic space, which has an abundance of hyperplanes. We crucially use a version of Helly's theorem for CAT(0) spaces. 
Our argument for small $p$ is largely inspired from \cite{BenSchTim-12-separation-graphs} where the separation profile of the real hyperbolic plane is computed. It consists of a symmetrisation argument. For large $p$, we construct for every finite set $A$, a $p$-Dirichlet function whose restriction to $A$ provides a good test function.  

\subsubsection{Lower bounds.} Obtaining lower bounds for groups with polynomial growth follows from well-known Poincar\'e inequalities in balls. It is interesting that the functional analytic interpretation of the separation profile gives us new estimates for nilpotent groups without effort. 

The lower bounds for hyperbolic Lie groups and small $p$ are obtained by considering parabolic closed nilpotent subgroups.  The cases of Bourdon--Pajot buildings, and of the cases when $p=Q$, are more interesting and more subtle.  For $p<Q$, we exploit the fact that their visual boundary satisfies (infinitesimal) Poincar\'e inequalities. We ``pull down" these inequalities on a sphere of large radius in the space using a discretization argument developed in the first part of the paper. 
To get the sharp lower bound in the $p=Q$ case, we use the Poincar\'e inequalities on spheres and a curve counting argument to find a new Poincar\'e inequalities on annuli.
A similar but simpler curve counting argument gives the lower bound for the $3$-regular tree, and hence all spaces in which it embeds.

\subsection{Structure of the paper}

The paper splits roughly into three parts. The first part introduces Poincar\'{e} profiles as monotone regular (in particular coarse) invariants. After introducing our notations and fixing the class of metric measure spaces under consideration, we present the more general definition of Poincar\'{e} constants in Section \ref{sec:Pconsts} and explain some basic properties. We then introduce Poincar\'{e} profiles and prove Theorem \ref{bthm:monotoneunderreg} in Sections \ref{sec:poinc-profiles} and \ref{sec:reg-maps-lseq} respectively.

The second part deals with relationships between Poincar\'{e} profiles. The descriptions of extremal profiles (Propositions \ref{bprop:P8growth} and \ref{bprop:P1sep}) and the connection with separation (Proposition~\ref{bprop:p1monotone}) are proved in Section \ref{sec:sepL1}, and the dependence on $p$ (Proposition \ref{bprop:pqmonotone}) is discussed in Section \ref{sec:pdep}.

The final part is dedicated to calculating profiles using the technology developed in the rest of the paper.  
Groups with polynomial growth (Theorem~\ref{bthm:CGLC-poly}) are considered in Sections~\ref{sec:polygrowth} and \ref{section:measureDimension}.
For hyperbolic spaces, trees (Theorem~\ref{bthm:trees}), lower bounds (Theorem~\ref{bthm:lbdQrgPI}) and upper bounds are in Sections~\ref{sec:trees}, \ref{sec:hyp-PI} and \ref{section:upperHyp} respectively, with applications (in particular, Theorem~\ref{bthm:hypconfdim}) discussed in Section~\ref{sec:applications}.

\subsection{Acknowledgements} We are grateful to Laurent Saloff-Coste and Anne Thomas for comments on earlier versions of this paper and for pointing out a number of related references, and to a referee for many very helpful suggestions.
%

\section{Notation and framework}\label{sec:noteframe}
We first introduce notation to be used throughout the paper.

Suppose $f, g: S \ra [0,\infty)$ where $S = \N$ or $S = [0, \infty)$.
We write $f\preceq_{u,v,\dots} g$ if there exists a constant $C>0$ depending only on $u,v,\dots$ such that $f(x)\leq Cg(x)$ for all $x \in S$.  If $f \preceq_{u,v,\dots} g$ and $g \preceq_{u,v,\dots} f$ then we write $f \asymp_{u,v,\dots} g$.  We drop the subscripts if the constants are understood.

We write $f\lesssim_{u,v,\dots} g$ if there exists a constant $C>0$ depending only on $u,v,\dots$ such that $f(x)\leq Cg(Cx+C)+C$ for all $x \in S$; similarly, we write $f \simeq_{u,v,\dots} g$ if $f \lesssim_{u,v,\dots} g$ and $g \lesssim_{u,v,\dots} f$.

Given a subset $A$ of a metric space $(X,d)$ and some $M\geq 0$ we define the closed $M$-neighbourhood of $A$ to be
\[
 [A]_M = \setcon{x\in X}{d(x,A)\leq M}.
\]
Given a point $x\in X$ and $r\geq 0$ we denote by $B(x,r)$ the closed metric ball of radius $r$ centred at $x$.

Let $(Z,\nu)$ be a measure space with positive finite measure. We denote the averaged integral by 
\[
 \dashint_Z f d\nu = \frac{1}{\nu(Z)} \int_Z f d\nu.
\]
Given a function $f\in L^p(X,\mu)$, another measure $\mu'$ such that $f\in L^p(X,\mu')$ and a measurable subset $Z\subseteq X$ we write
\[
 \norm{f}_{p,\mu'}=\left(\int_X \abs{f(z)}^p d\mu'(z)\right)^{\frac1p}\quad \textup{and}
 \]
 \[
 \norm{f}_{Z,p}=\left(\int_Z \abs{f(z)}^p d\mu(z)\right)^{\frac1p}.
\]
The $L^\infty$ norms $\norm{\cdot}_{\infty,\mu'}$ and $\norm{\cdot}_{Z,\infty}$ are defined analogously.

Given a graph $\Gamma=(V\Gamma,E\Gamma)$ and a subset $A\subset V\Gamma$, the full (or induced) subgraph of $\Gamma$ with vertex set $A$ is the graph with vertex set $A$ and edge set $\setcon{xy\in E}{x,y\in A}$.

The purpose of the remainder of this section is to introduce the class of spaces we will consider in this paper.

\begin{definition}\label{def:metricmeasurespace}
A \textbf{metric measure space} is a triple $(X,d,\mu)$ where $\mu$ is a non-trivial, locally finite, Borel measure on a complete, separable metric space $(X,d)$.
\end{definition}

The key examples are: graphs of bounded degree, Riemannian manifolds with bounded geometry and compactly generated locally compact groups, so we will make the following standing assumptions.

We will assume throughout the paper that any metric measure space $(X,d,\mu)$ satisfies the following properties:
\begin{itemize}
\item $X$ has \textbf{bounded packing on large scales}\footnote{If $r_0=0$, then we simply say that $X$ has bounded packing.}: there exists $r_0\geq 0$ such that for all $r\geq r_0$, there exists $K_r>0$ 
such that
$$\forall x\in X, \; \mu(B(x,2r))\leq K_r\mu(B(x,r)).$$
We then say that $X$ has \textbf{bounded packing on scales $\geq r_0$}.
\item $X$ is $k$\textbf{-geodesic} for some $k>0$: for every pair of points $x,y\in X$ there is a sequence $x=x_0,\ldots,x_n=y$ such that $d(x_{i-1},x_i)\leq k$ for all $i$ and $d(x,y)=\sum_{i=1}^n d(x_{i-1},x_i)$.
\end{itemize}

Up to rescaling the metric we will assume that $X$ is $1$-geodesic and has bounded packing on scales $\geq r_0=1$.

A subspace $Z\subset X$ is always assumed to be $1$\textbf{-thick} (a union of closed balls of radius $1$), so in particular it has positive measure. We equip $Z$ with the subspace measure and the induced $1$-distance
\[
 d(z,z')=\inf\set{\sum_{i=1}^n d(z_{i-1},z_i)}
\]
where the infimum is taken over all sequences $z=z_0,\ldots,z_n=z'$, such that each $z_i\in Z$ and $d(z_{i-1},z_i)\leq 1$.

Note that (as in the case of a disconnected subgraph) the induced $1$-distance will take values in $[0,\infty]$. 

\begin{remark} In the case of (the vertex set of) a bounded degree graph $X$, $d$ is the shortest path metric and $\mu$ is the (vertex) counting measure. Subspaces $Z$ are (vertex sets of) $1$-thick subgraphs equipped with the vertex counting measure and their own shortest path metric (the induced $1$-distance).

In a locally compact group $G$ with compact generating set $K$, we equip $G$ with a Haar measure (which is unique up to scaling) and the word metric $d=d_K$. 
\end{remark}
The reason for working with thick sets is justified by the following easy lemma (see \cite[Lemma 8.4]{Tes-Sobolev-ineq}).
\begin{lemma}\label{lem:doublingThick}
Assume $X$ has bounded packing on scales $\geq r_0$, and let $A\subset X$ be $r$-thick for some $r\geq r_0$. Then for all $u>0$,
$$\mu([A]_u)\preceq_u \mu(A).$$
\end{lemma}

\section{Poincar\'e constants}\label{sec:Pconsts}

Let $(X,d)$ be a metric space and let $a>0$. Given a measurable function $f:X\to \R$, we define its \textbf{upper gradient at scale} $a$ to be $$|\nabla_{a} f|(x)=\sup_{y,y'\in B(x,a)}|f(y)-f(y')|.$$ 
\begin{remark}
We have slightly modified the notation from \cite{Tes-Sobolev-ineq}, where the upper gradient was referred to as the ``local norm of the gradient'' and was denoted by $|\nabla f|_a$. The changes in this paper are for brevity; in what follows $\||\nabla_a f|\|_p$ will simply be denoted by $\|\nabla_a f\|_p$.
\end{remark}

\begin{definition}\label{mmspCheeger}
Let $(Z,d,\nu)$ be a metric measure space with finite measure and fix a scale $a>0$.
We define the $L^p$\textbf{-Poincar\'e constant at scale} $a$ of $Z$ to be
$$h_{a}^p(Z)=\inf_{f} \frac{\|\nabla_a f\|_p}{\|f\|_p},$$
where the infimum is taken over all $f\in L^p(Z,\nu)$ such that $f_Z:=\dashint_Z fd\nu=0$ and $f \not\equiv 0$. We adopt the convention that $h_{a}^p(Z)=0$ whenever $\nu(Z)=0$.
\end{definition}

Before continuing we list some basic properties of the Poincar\'{e} constant.
\begin{lemma}\label{lem:PCbasics} Let $(Z,d,\nu)$ be a metric measure space with finite measure.
\begin{enumerate}[(i)]
    \item \label{PCchangemetric} Let $\theta:(0,\infty)\to(0,\infty)$ be a non-decreasing function, and let $(Z',d',\nu')$ be a metric measure space such that $(Z,\nu)=(Z',\nu')$, and 
    $d'(z_1,z_2)\leq\theta(d(z_1,z_2))$ for all $z_1,z_2$. Then for all $a>0$, $$h^p_a(Z)\leq h^p_{\theta(a)}(Z').$$
	\item \label{PCchangemeasure} Let $(Z',d',\nu')$ be a metric measure space where $(Z,d)=(Z',d')$ and there exists some $M\geq 1$ such that $M^{-1}\nu(A)\leq\nu'(A)\leq M\nu(A)$ for every measurable $A\subseteq Z$. 
Then for all $a>0$, $$h_{a}^p(Z')\leq 2M^{2/p}h_{a}^p(Z).$$  
\end{enumerate}
\end{lemma}
\begin{proof}
Part (i) is immediate. For part (ii), let $f:X\to \R$ be a measurable function such that $\dashint f d\nu=0$ and let $m=\dashint fd\nu'$. We see that
$$\|f\|_{p,\nu} \leq 2\|f-m\|_{p,\nu}\leq 2M^{1/p}\|f-m\|_{p,\nu'}$$
The first inequality above is the $C=-m$ case of inequality \eqref{eq:mean-quasi-min} proved in Lemma \ref{lem:average/energymin}. 
On the other hand $$\|\nabla f\|_{p,\nu'}\leq M^{1/p}\|\nabla f\|_{p,\nu},$$
so we are done.
\end{proof}

To obtain a sensible definition of the $L^p$-Poincar\'{e} constant it is necessary to only consider functions whose average is zero and to choose a notion of gradient. In both cases there are multiple ways to do this.

\subsection{Choice of average}\label{sec:choiceav}
Given a measure space $(Z,\nu)$ with finite positive measure, there are multiple ways to define the ``average'' of a measurable function $f:(Z,\nu)\to\R$:
\begin{enumerate}
	\item the \textbf{average} $f_Z=\dashint_Z f d\nu$,
	\item a \textbf{median} $m_f$: any value such that $\nu(\set{f<m_f})\leq\nu(Z)/2$ and $\nu(\set{f>m_f})\leq\nu(Z)/2$,
	\item a $p$\textbf{-energy minimizer}: any value $c_p$ such that $\inf_c\norm{f-c}_p$ is attained for $c=c_p$.
\end{enumerate}

There is a simple comparison between the average and any energy minimizer, so choosing (1) or (3) gives comparable Poincar\'e constants.

\begin{lemma}\label{lem:average/energymin} Let $(Z,\nu)$ be a measure space with finite positive measure, and let $f:(Z,\nu)\to\R$ be a measurable function. For every $p\in[1,\infty)$ we have $\norm{f-c_p}_p\leq \norm{f-f_Z}_p\leq 2\norm{f-c_p}_p$.
\end{lemma}
\begin{proof} For any $C \in \R$ we have
\begin{equation}\label{eq:mean-quasi-min}
  \begin{split}
    \|f-f_Z\|_p & \leq \| f+C \|_p + \|C+f_Z \|_p
    \\ & = \|f+C\|_p + \nu(Z)^{1/p} \left| C+ \frac{1}{\nu(Z)}\int_Z f(z) d\nu(z) \right| \\
    & \leq \|f+C\|_p + \nu(Z)^{-1+1/p} \int_Z |C+f(z)| d\nu(z)
    \\ & \leq \|f+C\|_p + \nu(Z)^{-1+1/p} \| C+f \|_p \|1\|_{p/(p-1)}
    \\ & = 2 \|f+C\|_p.
  \end{split}
\end{equation}
In addition, if $C=-c_p$, then $\|f-c_p\|_p\leq \|f-f_Z\|_p$ by definition.
\end{proof}

In the case of $p=1$, this lemma combines with the following to show that taking either averages or medians will yield comparable Poincar\'{e} constants.

\begin{lemma}\label{lem:median-minimizer}
Let $(Z,\nu)$ be a measure space with finite positive measure $\nu$ and let $f:Z\to\R$ be a measurable function. Then a value $c$ is a $1$-energy minimizer $c_1$ of $f$ if and only if it is a median $m_f$.
\end{lemma}
\begin{proof}
	For $c'>c$, a calculation gives:
	\begin{multline}\label{eq:median-1}
		\|f-c'\|_1-\|f-c\|_1 = (c'-c)\big(\nu(\{f\leq c\})-\nu(\{f\geq c'\})\big) \\ +\int_{\{c<f<c'\}}(c+c'-2f)d\nu.
	\end{multline}
	
	If $c$ minimizes $\|f-c\|_1$, \eqref{eq:median-1} gives
	\[
		0 \leq 
		(c'-c)\big(\nu(\{f\leq c\})-\nu(\{f\geq c'\})\big)+(c'-c)\nu(\{c<f<c'\}).
	\]
	Letting $c' \ra c$, we get $\nu(\{f>c\}) \leq \nu(\{f\leq c\})$.
	The same argument applied to $-f$ gives $\nu(\{f<c\}) \leq \nu(\{f \geq c\})$,
	so $c$ is a median of $f$.

	Conversely, if $c$ is a median for $f$, \eqref{eq:median-1} gives
	\begin{align*}
		& \|f-c'\|_1-\|f-c\|_1 
		\\ & \quad = (c'-c)\left(\nu(\{f\leq c\})-\nu(\{f> c\})\right)
		\\ & \quad \qquad +(c'-c)\nu(\{c<f<c'\})
		+\int_{\{c<f<c'\}}(c+c'-2f)d\nu.
		\\ & \quad \geq (c'-c)(\tfrac{1}{2}\nu(Z) - \tfrac{1}{2}\nu(Z)) + \int_{\{c<f<c'\}}(2c'-2f) \geq 0,
	\end{align*}
	so increasing $c$ cannot lower $\|f-c\|_1$.  The same argument applied to $-f$
	gives that the median $c$ is also a minimizer for $\|f-c\|_1$.
\end{proof}

\begin{remark}\label{rmk:eval-comparison}
	For $\Gamma$ a finite graph of constant degree $d$, $\lambda_{1,p}(\Gamma)$, the first eigenvalue of the $p$-Laplacian on $\Gamma$, may be calculated to be the infimum of $\left(\sum_{xy\in E\Gamma} |f(x)-f(y)|^p\right)/ \left(\sum_{x\in V\Gamma} |f(x)-c_p(f)|^p d\right)$ over all non-constant $f$ with $c_p(f)$ the energy minimizer of $f$ (see \cite{Bou-12-lp-fixed-point}).  Thus by Lemma~\ref{lem:average/energymin} we have $\lambda_{1,p}(\Gamma) \asymp h^p(\Gamma)^{1/p}$.
\end{remark}

\subsection{Comparison with Lipschitz gradient}\label{sec:choicegrad}
Classical Poincar\'e inequalities on balls in $\R^n$ involve the $L^p$-norms of the usual gradient vector $\nabla f$.
For general metric spaces this makes no sense, but it is possible to define an analogue of the point-wise norm $|\nabla f|$.
Given this, one can define what it means for a metric measure space to satisfy a Poincar\'e inequality in this infinitesimal sense (see Section~\ref{sec:hyp-PI}).  
Let $(Z,d,\nu)$ be a metric measure space with finite (positive) measure. We define the Lipschitz gradient to be
\[
\Lip_x(f)= \limsup_{h\to 0} \sup_{y\in B(x,h)} \frac{\abs{f(x)-f(y)}}{h}.
\]
Given a metric space $(Z,d)$ we can define
\[
\hLip^p(Z) = \inf\frac{\norm{\Lip_x(f)}_p}{\norm{f}_p}
\]
where the infimum is taken over all non-constant Lipschitz functions $f:Z\to\R$ with average $0$.

Following \S 10.2 and \S 10.3 from \cite{Tes-Sobolev-ineq}, one can show that---under suitable assumptions on a metric measure space---the Poincar\'{e} constant relative to the Lipschitz norm (for Lipschitz functions) is equivalent to the Poincar\'{e} constant with respect to the gradient at some fixed scale $\alpha>0$.  A closely related result appears in \cite{Keith-Rajala}.

Here, we will focus on one direction (the only one required in the paper, namely in the proof of Theorem \ref{thm:hyp-PI-bdry}) which relies solely on a bounded packing assumption:
\begin{proposition}\label{prop:UpperGradToLargeScale}
Let $(Z,d,\nu)$ be a metric measure space with finite measure $\geq 1$, let $a>0$ and let $C \geq 1$.
Assume that for all $x\in Z$, $\nu(B(x,2a)) \leq C \nu(B(x,a/2))$.
Then, $$\hLip^p(Z)\preceq_{C,a,p}  h_{a}^p(Z).$$
\end{proposition}
\begin{proof}
We first need the following lemma:
\begin{lemma}\label{lem:Pconvolution}
Assume $h_{a}^p(Z)\leq 1/8$.
Let $(P_x)_{x}$ be a family of probability measures on $Z$, such that $P_x$ is supported in $B(x,a)$ for every $x\in Z$.
Then there exists $f\in L^{\infty}$ such that 
$$\frac{\|\nabla_{a}f\|_{p}}{\|Pf- (Pf)_Z\|_p}\leq 4h^p_{a}(Z),$$
where $Pf(x):=\int f dP_x.$
\end{lemma}

\begin{proof}
We start with $f$ with average $0$, $f_Z=0$, such that 
$$\frac{\|\nabla_{a}f\|_{p}}{\|f\|_p}\leq 2h_{a}^p(Z)\leq \frac14.$$
Observe that 
$$\|f-Pf\|_p\leq \|\nabla_{a} f\|_p\leq \tfrac14 \|f\|_p,$$
from which we deduce that 
$$|(Pf)_Z|=\left|\dashint_Z Pf\right|=\left|\dashint_Z Pf-f\right|
\leq \frac{\|f-Pf\|_p}{\nu(Z)^{1/p}}
\leq \frac{\|f\|_p}{4\nu(Z)^{1/p}}.$$
So $\|(Pf)_Z\|_p \leq \frac14 \|f\|_p$ and then we deduce by the triangle inequality that 
\[
\|Pf- (Pf)_Z \|_p\geq \frac{\|f\|_p}{2}. \qedhere
\]
\end{proof}
The rest of the proof of the proposition is similar to that of \cite[Theorem 10.9]{Tes-Sobolev-ineq}. For the convenience of the reader we sketch it. Define a $1$-Lipschitz map $\theta: Z\times Z\to \R_+$ by
$\theta(x,y)=d(y,B(x,a)^c)$. 
For $U \subset Z$ write
$$P_x(U)=\int_U \frac{\theta(x,y)}{K(x)} d\nu(y),$$ where
$K(x)=\int_{B(x,a)}\theta(x,z)d\nu(z)$.
Note that $K(x)\asymp_C \nu(B(x,a))$, and that by assumption, $\nu(B(x,a))\asymp_C \nu(B(y,a))$ as soon as $d(x,y) < a$.
Since $\theta$ is $1$-Lipschitz with respect to $x$, we see that for all $f\in L^{\infty}(Z)$,
$$\mathrm{\Lip}_x(Pf)\preceq_C |\nabla_{a} f|.$$
Note that if $h_{a}^p(Z)> \frac18$, then the statement of the proposition follows trivially. Hence we can assume that $h_{a}^p(X)\leq \frac18$.
By Lemma \ref{lem:Pconvolution} we deduce that there exists some function $f$ such that 
$$\frac{\|\mathrm{\Lip}_x(Pf)\|_p}{\|Pf-(Pf)_Z\|_p}\preceq_C h_{a}^p(Z).$$ 
Hence the proposition follows.
\end{proof}


\section{Poincar\'e profiles for metric measure spaces}\label{sec:poinc-profiles}

Our goal in this section is to generalise the Poincar\'{e} profile to the class of metric measure spaces defined in Section \ref{sec:noteframe}.

\begin{definition}\label{mmspPoincare}
Let $(X,d,\mu)$ be a metric measure space satisfying our standing assumptions, and fix some number $a\geq 2$. 
We define the \textbf{$L^p$-Poincar\'e profile} $\Lambda^p_{X,a}(r)$ of $X$ at scale $a$ to be the supremum of $\mu(A)h^p_{a}(A)$ over all subspaces $A \subset X$ satisfying $\mu(A)\leq r$. If no such subspace exists, define $\Lambda^p_{X,a}(r)=0$.
\end{definition}
Recall that by assumption, we only consider $1$-thick subsets of $X$ to be subspaces.

\begin{remark}\label{rem:NeumannDirich} As mentioned in the introduction, strictly speaking, we have defined the $L^p$-Neumann-Poincar\'{e} profile. We could alternatively define $L^p$-Dirichlet-Poincar\'{e} profiles using Dirichlet-Poincar\'{e} inequalities (considering the infimum over all functions which are $0$ on the boundary of $\Gamma$ in $X$, rather than those which have average $0$). As defined above, the monotone coarse invariant we obtain detects only if the space has infinite diameter. A small modification to the definition (taking the infimum of $\mu(A)h^p_{a}(A)$ over all subspaces $A \subset X$ satisfying $\mu(A)\geq r$) yields a coarse invariant (it is not even monotone under quasi--isometric embeddings) which detects isoperimetry (and in particular, F\o lner amenability) in the case $p=1$. Dirichlet-type Poincar\'e inequalities were introduced in \cite[Section 7.2]{Coul_randwalks} where they are called Sobolev inequalities (see also \cite{Tes-Sobolev-ineq} for a related notion of $L^p$-isoperimetric profile). They have been extensive studied in the cases $p=1$, where they are equivalent to isoperimetric inequalities, and $p=2$, where they govern the asymptotic behaviour of the probability of return of the simple symmetric random walk.
\end{remark}

We first prove that the Poincar\'e profile does not actually depend on the choice of $a$.

\begin{proposition}\label{prop:PoincareIndependant}
Assume that  $(Z,d,\nu)$ is a finite metric measure space. Then for all $a\geq 2$ and all $p\in[1,\infty)$ we have  
$$h^p_{a}(Z)\asymp_{a} h^p_{2}(Z).$$
\end{proposition}
\begin{proof}
We claim that for any $t \geq 0$, 
\begin{equation}\label{eq:levelsets}
\nu\left(\set{|\nabla_{a} f|\geq t}\right) \preceq_{a} \nu\left(\set{|\nabla_{2} f|\geq \frac{t}{5a}}\right),
\end{equation}
and $\nu(\set{|\nabla_{2} f|\geq t})\leq\nu(\set{|\nabla_{a} f|\geq t})$. Together these inequalities immediately imply the proposition.
The second inequality is obvious. Let $z\in Z$, and let $x,y\in B(z,a)$. Then one can easily check that our standing assumption implies that there exists a $1$-path $x=x_0,\ldots, x_n=y$ within $B(z,a)$ such that $n\leq 5a$. By the triangle inequality, this means that for at least one $1\leq i\leq n$, $|f(x_i)-f(x_{i-1})|\geq \frac{1}{5a}|f(x)-f(y)|$. Now for all $z'\in B(x_i,1)$ this implies that $|\nabla_{2} f|(z')\geq \frac{1}{5a}|f(x)-f(y)|$. Hence there is a $1$-thick subset which is $2a$-dense in the set $\{|\nabla_{a} f|\geq t) \}$ on which $|\nabla_{2} f|(z')\geq \frac{t}{5a}$. Thus, the left-hand inequality in (\ref{eq:levelsets}) follows from Lemma \ref{lem:doublingThick}.
\end{proof}

\begin{corollary}\label{cor:ProfileIndependant}
Assume that  $(X,d,\mu)$ satisfies our standing assumptions. Then for all $a,a'\geq 2$ and all $p\in[1,\infty)$ we have  
$$\Lambda^p_{X,a}\asymp_{a,a'} \Lambda^p_{X,a'}.$$
\end{corollary}
Moreover, by Lemma \ref{lem:PCbasics}, choosing a bi-Lipschitz equivalent metric and/or measure does not affect the $L^p$-Poincar\'e profile $\Lambda^p_{X,a}$ for sufficiently large $a$ (up to $\simeq$). In particular this means that for a compactly generated locally compact group, the $L^p$-Poincar\'e profile does not depend on the choice of Haar measure or on the choice of compact generating set.

In light of Corollary \ref{cor:ProfileIndependant}, we now refer to $\Lambda^p_{X}$ as the $L^p$-Poincar\'{e} profile of $X$, without the need to specify a scale.

\section{Regular maps and large scale equivalence}\label{sec:reg-maps-lseq}

The goal of this section is to prove Theorem \ref{bthm:monotoneunderreg}. Firstly, we formally introduce the notion of a coarse regular map and prove that the definition coincides with regular maps for bounded degree graphs.

\subsection{Regular maps}
In this section we show that Poincar\'{e} profiles are monotone non-decreasing under coarse regular maps. These maps are a natural adaptation of the regular maps considered in \cite{David-Semmes_fractal} to the context of metric measure spaces.

\begin{definition}\label{def:coarseregular}
  A map $F:(X,d,\mu)\to (X',d',\mu')$ is said to be \textbf{coarse regular} if it satisfies the following properties:
\begin{enumerate}[(i)]
\item $F$ is coarse Lipschitz: there exists an increasing function $\rho_+:[0,\infty)\to[0,\infty)$ such that for all $x,y\in X$, 
$$d(F(x),F(y))\leq \rho_+(d(x,y));$$
\item $F$ is coarsely measure preserving: there exists $\delta_0$ such that for all $\delta\geq \delta_0$ and for all ($1$-thick) subspaces $A\subset X$,
 $$\mu([A]_{\delta})\asymp_\delta\mu'([F(A)]_{\delta})\asymp_\delta \mu([F^{-1}(F(A))]_{\delta}).$$
  \end{enumerate}
  The {\it parameters} of $F$ are the constant $\delta_0$ as well as the function $\rho_+$.
\end{definition}

\begin{remark}\label{rem:compositionRegularMaps}
Coarse regular maps between spaces with bounded packing on large scales are stable under composition.
\end{remark}

In applications, coarse regular maps often are embeddings of the following kind.
\begin{definition}\label{def:largescaleEquiv}
%
 A coarse regular map $F:(X,d,\mu)\to (Y,d,\nu)$ is called a \textbf{large-scale embedding} if it is also a coarse embedding; there exists a function $\rho_-$ such that $\lim_{t\to \infty}\rho_-(t)=\infty$ and for all $x,y\in X$, 
 $$\rho_-(d(x,y))\leq d(F(x),F(y)).$$
 If, in addition, $[F(X)]_C=Y$ for some $C\geq 0$ (in other words, if $F$ is a coarse equivalence), then $F$ is called a \textbf{large-scale equivalence}.
 \end{definition}

It is easy to see that the relation ``there exists a large scale equivalence from $X$ to $Y$'' is an equivalence relation among metric measure spaces.

\begin{lemma}\label{lem:regcoarsereg} Let $X,X'$ be simplicial graphs of bounded degree equipped with the shortest path metrics and vertex counting measures. A map $F:VX\to VX'$ is regular in the sense of \cite{BenSchTim-12-separation-graphs} if and only if it is coarse regular as a map between metric measure spaces.
\end{lemma}
\begin{proof} Recall that $F$ is regular if there is a constant $K$ such that $d_{X'}(F(x),F(y))\leq K(1+d_X(x,y))$ and for every $y\in VX'$, $F^{-1}(B(y,1))$ is contained in a union of at most $K$ balls of radius $1$ in $X$. The first of these conditions immediately implies that $F$ is coarse Lipschitz in the sense of Definition \ref{def:coarseregular}. It remains to show that $F$ is coarsely measure preserving. 

Since we are working with counting measures, the image of every set of measure $m$ has measure at most $m$ and the pre-image of every set of measure $m$ is contained in a union of at most $Km$ balls of radius $1$ in $X$, so has measure at most $K(\Delta_{X}+1)m$ where $\Delta_{X}$ is the maximal vertex degree of the graph $X$. Since $X$ and $X'$ have bounded degree, $F$ is coarse regular.

Suppose $F$ is coarse regular, then it is $\rho_+(1)$-Lipschitz. Fix some suitable $\delta_0$, let $x'=F(x)\in F(VX)$ and notice that the ($1$-thick) subspace $A=[x]_1$ satisfies 
\[\abs{F^{-1}(x')}\leq\abs{[F^{-1}(F(A))]_{\delta}}\preceq_\delta \abs{[A]_\delta} \leq \abs{[x]_{\delta+1}}\preceq_{\delta, \Delta_X} 1.
\]
Therefore
\[
\abs{F^{-1}(B(x',1))} \leq \sum_{x'' \in B(x',1)} \abs{F^{-1}(x'')} \preceq_{\delta,\Delta_X} \sum_{x'' \in B(x',1)} 1 \preceq_{\Delta_X} 1.
\]
In particular, $\abs{F^{-1}(B(x',1))}$ can be covered by a uniformly bounded number of balls of radius $1$ in $X$. Thus, $F$ is regular.
\end{proof}

The following proposition is the main goal of this section, and will be proved in \S \ref{sec:coarseReg}. It is the natural generalisation of \cite[Lemma 1.3]{BenSchTim-12-separation-graphs} from separation profiles of graphs to metric measure spaces and Poincar\'e profiles.

\begin{proposition}\label{prop:Poincarecoarsereg}
Let $F:X\to X'$ be a coarse regular map between metric measure spaces which satisfy our standing assumptions. Then for all $p\in[1,\infty)$,
$$\Lambda^p_{X}\lesssim_p \Lambda^p_{X'}.$$
\end{proposition}
Theorem~\ref{bthm:monotoneunderreg} follows immediately from Lemma~\ref{lem:regcoarsereg} and this proposition.
Note that by Proposition~\ref{prop:PoincareIndependant} it suffices to prove $\Lambda^p_{X,a}\lesssim_p \Lambda^p_{X',a'}.$ for some $a,a'\geq 2$.

An important consequence of Proposition~\ref{prop:Poincarecoarsereg} is the following.

\begin{proposition}\label{prop:latticeseparation}
Let $G$ and $H$ be compactly generated locally compact groups, and let $\phi:H\to G$ be a proper continuous morphism (i.e.\ $\ker \phi$ is compact and $\phi(H)$ is a closed subgroup). We assume that both $G$ and $H$ are equipped with left-invariant Haar measures and word metrics with respect to some compact symmetric generating sets. Then, for all $p\in[1,\infty)$, $\Lambda^p_H\lesssim_p\Lambda^p_G$. If $\phi(H)$ is co-compact then $\Lambda^p_H\simeq_p\Lambda^p_G$.
\end{proposition}
\begin{proof} The morphism $\phi$ is a large-scale embedding hence it is coarse regular. If $\phi(H)$ is co-compact then $\phi$ is a large-scale equivalence. The result then follows from Proposition \ref{prop:Poincarecoarsereg}.
\end{proof}

\subsection{Proof of Theorem $\ref{bthm:monotoneunderreg}$}\label{sec:coarseReg}

The argument behind the proof is as follows: given a coarse regular map $F:X\to X'$ which is $\rho_+$-coarse Lipschitz and coarsely measure preserving for all $\delta\geq\delta_0$, and a subspace $Z\subseteq X$, we define $M=\max\set{\rho_+(1),\delta_0}$ and build metric measure space discretizations $Y$ of $Z$ and $Y'$ of the $1$-thick subspace $[[F(Z)]_{M}]_1$. By the definition of a coarse regular map and Lemma \ref{lem:doublingThick}, 
\[
\mu_X(Z)\asymp_{M}\mu_X([Z]_{M})\asymp_{M}\mu_{X'}([[F(Z)]_{M}]_1).
\]
We then show that the process of taking a discretization yields spaces with equal measure and comparable Poincar\'{e} constants, and finally prove that $Y$ and $Y'$ have comparable Poincar\'e constants.

The first step of the proof consists in constructing discretizations of our spaces. 
We fix some $b\geq M$ (which we refer to as the \textit{discretization parameter}). We let $Y\subset Z$ be a maximal $3b$-separated subset of $Z$. By maximality $Z$ is covered by the union of balls $\bigcup_{y\in Y}B(y,9b)$. We pick  a set $A_y$ for each $y\in Y$ with the following properties:
\begin{itemize}
\item $B(y,b)\subset A_y\subset B(y,9b)$;
\item $(A_y)_{y\in Y}$ forms a measurable partition of $Z$.
\end{itemize} 
We equip $Y$ with the subspace distance and the measure $\nu_Y(y)=\nu(A_y)$.
Let $\pi: Z\to Y$ be defined by ``$\pi(z)$ is the only $y\in Y$ such that $z\in A_y$''. Note that $\pi$ is surjective, and a right-inverse of the inclusion $j: Y\to Z$. Moreover, $\pi^{-1}(y)=A_y$ for every $y\in Y$.

\begin{remark}\label{rem:YpiAndj}
Observe that the choice of $b$ ensures that $Y$ has bounded packing at all scales $\geq 0$, and that both  $\pi$ and $j$ are large-scale equivalences. In particular, if $Y'$ is a similar discretization of $[[F(Z)]_{M}]_1$, then $\Psi=\pi'\circ F\circ j$ is a coarse regular map. Moreover, if one chooses the discretization parameter $b'$ large enough, then $\Psi$ is surjective.
\end{remark}

Our next goal is to compare the Poincar\'e constant of a subspace with that of its discretization.

\begin{lemma}\label{lem:disc}
Let $(Z,d,\nu)$ be a metric measure space with finite measure. Suppose $(Y,d,\nu_Y)$ is a discretization (with parameter $b\geq 2$) of $Z$ as above. Then for all $a\geq b$,
$$h_{a}^p(Y)\lesssim_a h_{20a}^p(Z),\quad\textup{and}\quad h_{a}^p(Z)\leq h_{20a}^p(Y).$$
\end{lemma}
\begin{proof}
Let $f\in L^{\infty}(Z)$ be such that $\dashint_Z fd\nu=0$. We define $\phi\in \ell^{\infty}(Y)$ by 
$\phi(y)=\dashint_{A_y}fd\nu$. Clearly $\dashint_Y \phi d\nu_Y=0$ and $\|\phi\circ \pi\|_{Z,p} = \|\phi\|_{Y,p}$. 
Write $f(z)=\phi(\pi(z))+\dashint_{A_{\pi(z)}}(f(z)-f(w))d\nu(w)$.
Then
\begin{align*}
  \|f\|_{Z,p} & \leq \|\phi\circ\pi\|_{Z,p} + \left(\int_Z\left|\dashint_{A_{\pi(z)}}\left(f(z)-f(w)\right) d\nu(w)\right|^p d\nu(z)\right)^{1/p}
  \\ & \leq \|\phi\|_{Y,p} + \left(\int_Z \dashint_{A_{\pi(z)}} |f(z)-f(w)|^p\,d\nu(w)d\nu(z)\right)^{1/p}
  \\ & \leq \|\phi\|_{Y,p} + \left(\int_Z |\nabla_{10a} f|(z)^p \right)^{1/p}
  \\ & = \|\phi\|_{Y,p} + \|\nabla_{10a} f \|_p\ .
\end{align*}
On the other hand, it is immediate from the definitions that $ |\nabla_{a} \phi|(y)\leq |\nabla_{20a}f|(z)$ for all $z\in A_y$. 

We now prove the first inequality. 
If $h_{20a}^p(Z)\leq \frac12,$ then 
for any $\epsilon\in(0,1/6)$ we can find $f$ as above so that
\[
  \frac{2}{3}\geq \frac{1}{2} +\epsilon \geq h_{20a}^p(Z)+\epsilon
  \geq \frac{\|\nabla_{20a} f\|_p}{\|f\|_p}
  \geq \frac{\|\nabla_{20a} f\|_p}{\|\phi\|_p + \|\nabla_{20a} f\|_p}.
\]
Thus $\|\nabla_{20a} f\|_p \leq 2\|\phi\|_p$ and
\[
	h_{20a}^p(Z)+\epsilon \geq \frac{\|\nabla_a \phi\|_p}{3\|\phi\|_p}
	\geq \frac{1}{3} h_a^p(Y).
\]
Since $\epsilon$ was arbitrary, $h_{a}^p(Y)\leq 3h_{20a}^p(Z)$. Moreover, it is easy to see that  $h_{a}^p(Y)\lesssim_a 1$ (a much more general statement is proved in Proposition \ref{prop:linubd}), so if $h_{20a}^p(Z)\geq \frac12$, then $h_{a}^p(Y)\lesssim_a h_{20a}^p(Z)$.

The other direction is easier: 
given $\psi\in \ell^{\infty}(Y)$, such that $\dashint_Y \psi d\nu_Y=0$ we define $g=\sum_{y\in Y} \psi(y)1_{A_y}$, where $1_{A_y}$ denotes the characteristic function of $A_y$. We clearly have $\dashint g d\nu=0$ and $\|g\|_p=\|\psi\|_p$. 
Hence we are left with comparing the gradients. 
\begin{align*}
\|\nabla_{r}g \|_p^p & =  \sum_Y\nu(A_y)\dashint_{A_y}\sup_{z',z''\in B(z,a)}|g(z')-g(z'')|^pd\nu(z) \\
               & \leq   \sum_Y\nu(A_y)\sup_{z',z''\in B(y,10a)}|g(z')-g(z'')|^p \\
               & \leq    \sum_Y \nu_Y(y) \sup_{y',y''\in B(y,20a)\cap Y} |\psi(y')-\psi(y'')|^p\\
               & =   \|\nabla_{20a}\psi\|_p^p. \qedhere
\end{align*}
\end{proof}
Now we compare the Poincar\'e constants of discrete spaces related by a sufficiently nice surjective coarse regular map.

\begin{lemma}\label{lem:regularDiscrete}
Let $\pi:(Y,d,\nu)\to (Y',d,\nu')$ be a map between two discrete metric measure spaces with finite (non-degenerate) measures, and assume that:
 \begin{itemize}
\item $\pi$ is surjective;
\item $\nu(\pi^{-1}(y'))=\nu'(y')$.
\end{itemize}
Then for all $a\geq 0$ and $C$ such that $d(y,z)\leq a$ implies $d(\pi(y),\pi(z))\leq Ca$, we have
$$h_{a}^p(Y)\leq h_{Ca}^p(Y').$$
\end{lemma}
\begin{proof}
Choose $f'\in \ell^{\infty}(Y')$ such that $\int f'd\mu'=0$ and let $f=f'\circ \pi$. Clearly, we have $\int fd\mu =0$, and $$\|f\|_p=\|f'\|_p.$$
Moreover, for every $y\in Y$, if $y_1,y_2\in B(y,a)$ then $\pi(y_1),\pi(y_2)\in B(\pi(y),Ca)$. So a straightforward computation shows that 
\[ \|\nabla_{a} f\|_p\leq \|\nabla_{Ca} f'\|_p. \qedhere \]
\end{proof}

As a result we obtain a version of Proposition~\ref{prop:Poincarecoarsereg} in the uniformly discrete case.

\begin{corollary}\label{cor:discreteCase}
Let $\Psi:(Y,d,\nu)\to (Y',d,\nu')$ be a surjective coarse regular map between uniformly discrete spaces, which have bounded packing at any scale. Then, for all $a>0$, there exists $C>0$ such that 
$$h_{a}^p(Y)\preceq_a h_{Ca}^p(Y').$$
\end{corollary}
\begin{proof}
The assumptions imply that $\Psi$ is coarse Lipschitz, surjective, and such that  $\nu(\Psi^{-1}(y'))\asymp_a  \nu'(y')$. Hence the corollary follows from Lemmas \ref{lem:PCbasics}\ref{PCchangemeasure} and \ref{lem:regularDiscrete}: if we push the measure $\nu$ forward with $\Psi$ to obtain a measure $\Psi_*\nu$ on $Y'$ we have
\[
  h_a^p(Y,\nu) \preceq h_{Ca}^p(Y',\Psi_*\nu) \preceq h_{Ca}^p(Y',\nu').\qedhere
\]
\end{proof}

Combining these results we are in a position to prove Proposition \ref{prop:Poincarecoarsereg}.

\begin{proof}[Proof of Proposition \ref{prop:Poincarecoarsereg}]
Let $Z$ be a $1$-thick subspace of $X$ and define $Z'=[[F(Z)]_M]_1$ where $M=\max\set{\delta_0,\rho_+(1)}$. Then $Z'$ is a $1$-thick subspace of $X'$ and $\mu(Z)\asymp_{M}\mu'(Z')$.
Let $b,b'$ be sufficiently large that the discretizations $Y$ of $Z$ and $Y'$ of $Z'$ satisfy the hypotheses of Lemma \ref{lem:disc} for some suitable $a=a(b,b')\geq 2$ and so that $\Psi=\pi'\circ F \circ j$ is surjective. Note that $b$ and $b'$ may be chosen independently of the choice of subspace $Z$ of $X$, hence $a$ does not depend on $Z$.

Applying Corollary \ref{cor:discreteCase} we see that there exists a constant $C$ depending only on $a$ such that $h^p_{a}(Y)\preceq_a h^p_{Ca}(Y')$. Now, by Lemma \ref{lem:disc} $h^p_a(Z)\preceq_{a,M} h^p_{C'a}(Z')$ where $a,M,C'$ do not depend on $Z$.

Thus there is some $M'$ depending only on $M$ and $Y'$ such that
\begin{eqnarray*}
\Lambda^p_{X,a}(r) & = & \sup\setcon{\mu(Z)h^p_a(Z)}{\mu(Z)\leq r}\\ & \lesssim_{a,M} & \sup\setcon{\mu'(Z')h^p_{C'a}(Z')}{Z'=[[F(Z)]_M]_1,\ \mu(Z)\leq r} \\ & \leq & \Lambda^p_{X',C'a}(M'r).
\end{eqnarray*}
We conclude using Corollary \ref{cor:ProfileIndependant}.
\end{proof}

\section{Extremal profiles: growth and separation}\label{sec:sepL1}

\subsection{Growth and the $L^\infty$-Poincar\'{e} profile}
In this section we give the proof of Proposition \ref{bprop:P8growth}. Recall our standing assumptions: a metric measure space $(X,d,\mu)$ is $1$-geodesic and has bounded packing at scales $\geq r_0=1$. Recall also that the growth function $\gamma_X(k)$ is the supremum of the measures of balls of radius $k$ in $X$, and the inverse growth function $\kappa_X(n)$ is the infimal $s$ such that there exists a ball $B\subset X$ of radius $s$ with measure $>n$. By assumption subspaces are $1$-thick and equipped with a $1$-geodesic metric.

\begin{proposition}\label{prop:Linfty_inverse growth} Let $(X,d,\mu)$ be a metric measure space with unbounded growth function $\gamma_X:[1,\infty) \ra (0, \infty)$, and let $a\geq 3$. Then
\[
\Lambda^\infty_{X,a}(r) \simeq_a \sup\setcon{\frac{s}{\kappa_X(s)}}{\gamma_X(1) \leq s\leq r},
\]
where we interpret $\sup \emptyset$ to be $0$.
\end{proposition}

In all our applications, the function $\sup\setcon{\frac{s}{\kappa_X(s)}}{\gamma_X(1) \leq s\leq r}$ will be equivalent to $\frac{r}{\kappa_X(r)}$ but in general this may not be the case. The proof requires a lemma.

\begin{lemma}\label{lem:inftylwbd} Let $Z$ be a subspace of $X$ with diameter $m\geq 3$ and let $a\geq 3$. Then $h^\infty_a(Z)\leq \frac{12a}{m}$, and if every $y,z\in Z$ can be joined by a $1$-path of length $\leq 2m$ then $h^\infty_a(Z) \geq \frac{1}{2m}$.
\end{lemma}
\begin{proof}
Choose $x,y\in Z$ such that $d(x,y)\geq m-\delta$, and define $f(z)=d(x,z)$.  It is clear that $f(z)\leq 1$ and $f(z')\geq m-1-\delta$ hold whenever $z\in B(x,1)$ and $z'\in B(y,1)$ and balls of radius $1$ have positive measure by the bounded packing assumption, so $\norm{f-f_Z}_\infty\geq \frac{m-2-\delta}{2}\geq \frac{m-\delta}{6}$, while $\norm{\nabla_a f}_\infty\leq 2a$ by the triangle inequality. Thus $h^\infty_a(Z)\leq \frac{12a}{m-\delta}$ for all $\delta>0$.

For the second inequality, fix $\delta>0$ and let $f\in L^\infty(Z)$ satisfy $\inf_{z\in Z} f(z)=0$. Choose $y,z$ so that $(f(z)-f(y))+\delta\geq \sup_{z\in Z}{\abs{f(z)}}\geq\norm{f}_\infty$.

By our hypothesis there exists a sequence of points $y=z_0,\ldots,z_k=z$ such that $k\leq 2m$ and $d(z_i,z_{i+1})\leq 1$ for all $i$. Since $Z$ is $1$-thick, for each $i$ we can choose $y_i$ so that $d(z_i,y_i)\leq 1$ and $B(y_i,1)\subseteq Z$. There is some $j$ such that $\abs{f(z_j)-f(z_{j-1})}\geq \frac{1}{2m}(\norm{f}_\infty-\delta)$, hence $\nabla_3 f \geq \frac{1}{2m}(\norm{f}_\infty-\delta)$ holds on $B(y_j,1)$ which has positive measure. Therefore, $\norm{\nabla_3 f}_\infty\geq \frac{1}{2m}(\norm{f}_\infty-\delta)$. Since we have $\norm{f-f_Z}_\infty\leq \norm{f}_\infty$, letting $\delta\to 0$, we see that $h^\infty_a(Z)\geq \frac{1}{2m}$.
\end{proof}

\begin{proof}[Proof of Proposition \ref{prop:Linfty_inverse growth}]
The upper bound on $\Lambda^\infty_{X,a}(r)$ follows immediately from Lemma \ref{lem:inftylwbd}.  Indeed, if $\mu(Z) \leq r$ then 
\[
	\mu(Z)h_a^\infty(Z) \lesssim_a \frac{\mu(Z)}{\diam(Z)} \leq \frac{\mu(Z)}{\kappa(\mu(Z))},
\]
so if $\mu(Z) \geq \gamma_X(1)$ we are done.
We can ignore $Z$ with $\mu(Z)$ bounded by any fixed constant like $\gamma_X(1)$ since any $f \in L^\infty(Z)$ has a representative with $\| \nabla_a f \|_\infty \leq 2\|f\|_\infty$, and so $\mu(Z)h_a^\infty(Z) \leq 2 \mu(Z)$ is bounded too.

We now prove the lower bound.
Let $t\geq 2$ and choose $x_t\in X$ such that $\mu(B(x_t,t))\geq \frac12\gamma_X(t)$. Define $Z_t$ to be the $1$-thick subspace $[B(x_t,t-1)]_1$. By Lemma \ref{lem:doublingThick} there is a constant $C$ (which does not depend on $t$) such that $
 \mu(Z_t) \leq \mu(B(x_t,t)) \leq \gamma_X(t) \leq C\mu(Z_t)$.

By Lemma \ref{lem:inftylwbd} $h^\infty_a(Z_t)\in[\frac{1}{12t},\frac{2a}{(t-1)}]$, so $\mu(Z_t)h^\infty_a(Z_t)\asymp_a \frac{\gamma_X(t)}{t}$.

There exists $s_0 \geq \gamma_X(1)$ so that for all $s \geq s_0$, $\kappa_X(s) \geq 3$.
On any bounded interval in $[\gamma_X(1),\infty)$, $\kappa_X$ is $\geq 1$ and so $s/\kappa_X(s)$ is bounded,
thus we may assume that $s$ and $r$ satisfy $s_0 \leq s \leq r$.
Repeating the above argument, we see that $\gamma_X(t)/\gamma_X(t-1)$ has a uniform upper bound which is independent of $t$.
Let $t = \kappa_X(s)-1 \geq 2$, and so $\gamma_X(t) \leq s \leq \gamma_X(t+2) \preceq \gamma_X(t)$. Thus for $r \geq s_0$,
\[
\Lambda_{X,a}^\infty(r) \succeq_a \frac{\gamma_X(t)}{t} \succeq \frac{s}{\kappa_X(s)}. \qedhere
\]
\end{proof}

\subsection{Separation profiles of metric measure spaces}\label{sec:sepdefns}
The first author has shown that the separation of a graph has an equivalent formulation in terms of Cheeger constants rather than cut sizes~\cite{HumSepExp}.  We now extend this to the setting of metric measure spaces $(X,d,\mu)$ which are $1$-geodesic and have bounded packing at scales $\geq 1$.

Given a subspace $A\subset X$ (which as usual we assume is $1$-thick and equipped with the induced measure and induced $1$-geodesic metric) we define the \textbf{boundary at scale} $a\geq 1$ of $A$ to be $$\partial_a A=[A]_a\cap
[A^c]_a$$ with the usual notation $A^c = X \setminus A$. For clarity, given a subspace $Z$ of $X$ and $A\subset Z$, we also define the boundary at scale $a$ of $A$ in $Z$ to be $\partial^Z_a A=Z\cap[A]_a\cap[Z\setminus A]_a$, and use this notion in the following.

\begin{definition}\label{def:mmsCheeger}
Let $(Z,d,\nu)$ be a metric measure space, where $\nu(Z)$ is finite and let $a\geq 2$. We define the \textbf{Cheeger constant at scale} $a$ of $Z$ to be
$$h_a(Z)=\inf \setcon{\frac{\nu(\partial^Z_a \Omega)}{\nu(\Omega)} }{ \nu(\Omega)\leq \frac{\nu(Z)}{2} }.$$

Let $(X,d,\mu)$ be a metric measure space. We define the function $\sep_{X,a}(r)=\sup \set{\mu(Z)h_a(Z)}$, where the supremum is taken over all ($1$-thick) subspaces $Z\subseteq X$ with $\mu(Z)\leq r$, and is $0$ if no such subspaces exist.
\end{definition}

\begin{remark}\label{rmk:bdrydefs} If $\Gamma$ is a finite graph of bounded degree $D$ then the boundary at scale $a$ has comparable size to the vertex boundary, so the usual (vertex) Cheeger constant $h(\Gamma)$ satisfies $h(\Gamma)\asymp_{a,D} h_a(\Gamma)$. As a result, if $X$ is an infinite graph of bounded degree $D$, then $\sep_{X,a}\simeq_{a,D}\sep_X$, where $\sep_X$ is the usual separation function for graphs.  (See~\cite[Propositions 2.2, 2.4]{HumSepExp}.)
\end{remark}

\subsection{Comparing Cheeger and $L^1$-Poincar\'{e} constants}\label{section:Cheeger/L1}

Our next goal is to prove Proposition \ref{bprop:P1sep}. Along the way we will also prove Proposition \ref{bprop:p1monotone}.

\begin{proposition}\label{prop:Lambda1Sep}
Let $(X,d,\mu)$ be a metric measure space and let $a\geq 2$. Then  
$$\frac{1}{2} \sep_{X,a} \leq\Lambda^1_{X,a}\leq \sep_{X,a}.$$
\end{proposition}

We prove this by comparing the Cheeger constant and the $L^1$-Poin\-car\'{e} constant.  We recall the following classical co-area formula.

\begin{proposition}\label{prop:coarea} 
Let $(X,d,\mu)$ be a metric measure space. The following co-area formula holds for every non-negative measurable function $f:X\to\R$.
\begin{equation}\label{coaire}
\int_{X}|\nabla_a f|(x)d\mu(x) = \int_{\R_+}\mu\left(\partial_a \{f> t\}\right)dt
\end{equation}
\end{proposition}
\begin{proof}
For every measurable subset $A\subset X$, we have
\begin{equation}\label{eq:PartialGradient}
\mu(\partial_a A)=\int_X|\nabla_a 1_A|(x)d\mu(x).
\end{equation}
Thus, (\ref{coaire}) follows by integrating over $X$ the following
local equalities
\begin{equation}\label{localcoaire}
|\nabla_a f|(x) = \int_{\R_+}|\nabla_a 1_{\{f> t\}}|(x) dt.
\end{equation}
It remains to show that these equalities hold for all $x\in X$.

Notice that $|\nabla_a 1_{\{f>t\}}(x)|=1$ if and only if
there exists $y,y' \in B(x,a)$ with $f(y)>t$ and $f(y')\leq t$.
In particular, $|\nabla_a 1_{\{f>t\}}(x)|$ equals one for
$t \in (\inf_{B(x,a)}f,\sup_{B(x,a)}f)$ and equals zero for
$t \notin [\inf_{B(x,a)}f,\sup_{B(x,a)}f]$.
Hence, $$\int_{\R_+}|\nabla_a 1_{\{f> t\}}|(x) dt =
\sup_{B(x,a)}f-\inf_{B(x,a)}f = |\nabla_a f|(x),$$ which proves
(\ref{localcoaire}).
\end{proof}

We can now prove the required relation between $h_a(Z)$ and $h_a^1(Z)$.

\begin{proposition}\label{prop:h1h}
Let $(Z,d,\nu)$ be a metric measure space with finite positive measure $\nu$ and let $a\geq 2$. Then  
$$h^1_a(Z)\leq h_{a}(Z)\leq 2h^1_a(Z).$$
\end{proposition}
\begin{proof}
Let $\Omega\subset Z$ such that $\nu(\Omega)\leq \nu(Z)/2$. We deduce from  (\ref{eq:PartialGradient}) that
$$\|\nabla_a f\|_1=\nu(\partial^Z_a \Omega),$$
where $f=1_{\Omega}$. 
On the other hand,
$$\|f-f_Z\|_1=\nu(\Omega)\left(1-\frac{\nu(\Omega)}{\nu(Z)}\right)+(\nu(Z)-\nu(\Omega))\left(\frac{\nu(\Omega)}{\nu(Z)}\right)\geq \nu(\Omega).$$
Hence $h_{a}^1(Z)\leq h_a(Z)$.

By Lemmas \ref{lem:average/energymin} and \ref{lem:median-minimizer}, for each $\delta>0$ we may choose $f\in L^1(Z,\nu)$ (with median $0$) such that
$$ \frac{\|\nabla_a f\|_1}{\|f\|_1}\leq 2 h_a^1(Z)+\delta.$$
Let $f_+ = \max\{f,0\}$ and $f_- = \min\{f,0\}$.
For any $s,s',t,t' > 0$ if $\frac{s+s'}{t+t'} \leq C$ then $\frac{s}{t} \leq C$ or $\frac{s'}{t'}\leq C$.
Since $\|f\|_1=\|f_-\|_1+\|f_+\|_1$ and $\|\nabla_a f\|_1=\|\nabla_a f_+\|_1+\|\nabla_a f_-\|_1$, we deduce that up to replacing $f$ by $-f$, we have 
$$ \frac{\|\nabla_a f_+\|_1}{\|f_+\|_1}\leq 2 h_a^1(Z) +\delta.$$
Hence using (\ref{coaire}) and the fact that 
$$\|f_+\|_1=\int_{\R_+}\nu(\{f>t\})dt,$$ we conclude that there exists some $t\geq 0$ such that the subset $\Omega_t=\{f> t\}$ satisfies
$$h_a(Z)\leq \frac{\nu(\partial^Z_a\Omega_t) }{\nu(\Omega_t)}\leq 2h_a^1(Z)+\delta.$$  
This proves the second inequality.
\end{proof}

\begin{proof}[Proof of Proposition \ref{bprop:P1sep}:] By Remark \ref{rmk:bdrydefs} and Proposition \ref{prop:Lambda1Sep} for all $a\geq 2$:
\[
 \sep_X(r) \simeq_a \sep_{X,a}(r) \simeq \Lambda^1_{X,a}(r). \qedhere
\]
\end{proof}

\begin{proof}[Proof of Proposition \ref{bprop:p1monotone}:] The first half of the above proof can easily be adapted to prove that $2^{1-p}h^p_a(Z)^p\leq h_{a}(Z)$. Hence, $h_a^p(Z)^p\leq 2^ph_a^1(Z)$.
\end{proof}

\section{Dependency on $p$}\label{sec:pdep}

In this section we prove Proposition~\ref{bprop:pqmonotone}.
One trivial upper bound can always be put on Poincar\'{e} constants.

\begin{proposition}\label{prop:linubd}
Let $(Z,d,\nu)$ be a metric measure space with $\nu(Z)$ finite. Assume there is no $z\in Z$ with $\nu(\set{z})>\frac23\mu(Z)$. For all $p\in[1,\infty)$ and all $a\geq 2$, $h^p_a(Z)\leq 2 \cdot 3^{\frac1p}$.
\end{proposition}
\begin{proof} By our standing assumptions (Definition~\ref{def:metricmeasurespace}), $\nu$ is measure isomorphic to a real interval and an at-most-countable collection of atoms.  It is then easy to find a subset $Y\subset Z$ satisfying $\frac13\nu(Z)\leq \nu(Y) \leq \frac23\nu(Z)$.  Let $f$ be the characteristic function of $Y$. 

Then $\norm{f-f_Y}_p^p \geq \frac{\nu(Z)}{3 \cdot 2^{p}}$ and $\norm{\nabla_a f}_p^p \leq \nu(Z)$, thus $h^p_a(Z) \leq 2 \cdot 3^{\frac1p}$.
\end{proof}

Equipped with this we are now able to study the relationship between different Poincar\'{e} profiles of the same space and prove Proposition \ref{bprop:pqmonotone}.

\begin{proposition}\label{prop:monotonep} 
Let $(Z,d,\nu)$ be a metric measure space with $\nu(Z)$ finite. Assume there is no $z\in Z$ with $\nu(\set{z})>\frac23\nu(Z)$. Then for all $1\leq p \leq q<\infty$ and all $a\geq 2$,
$$h_a^q(Z)\succeq_{p,q} h_a^p(Z).$$

For all metric measure spaces $(X,d,\mu)$ (where $\mu$ is possibly infinite), and all $1\leq p \leq q<\infty$,
 $$\Lambda^q_X\succeq_{p,q} \Lambda^p_X.$$
\end{proposition}
\begin{proof}
Our goal is to prove that for any function $g:Z\to\R$, there is a function $f:Z\to\R$ such that
\[
 \frac{\norm{\nabla_a g}_q}{\norm{g-g_Z}_q} \succeq_{p,q} \frac{\norm{\nabla_a f}_p}{\norm{f-f_Z}_p} \geq h_a^p(Z).
\]
Taking the infimum over all $g$ would then yield the desired result. From this, we see that it suffices to consider all functions $g$ which satisfy the upper bound $\norm{\nabla_a g}_q\leq 6\norm{g-g_Z}_q$ given by Proposition \ref{prop:linubd}.  By \eqref{eq:mean-quasi-min} we have that for all $C \in \R$, $6\norm{g-g_Z}_q\leq 12\norm{g-C}_q$.

For $a \in \R$ and $p \geq 1$, write $\{a\}^p = \text{sign}(a) |a|^p$.
For each $C$, define $f^C:Z\to\R$ by $f^C(z)=\{g(z)+C\}^{q/p}$, for some $C \in \R$.
Since $f^C_Z$ is a continuous function of $C$, we fix $C$ so that $f^C_Z=0$. Set $f=f^C$.

For each $z\in Z$ let $\overline{(g+C)}_a(z)=\sup\varsetcon{\abs{g(z')+C}}{d(z,z')\leq a}$.

By the mean value theorem (see e.g.\ Matou\v{s}ek~\cite[Lemma 4]{Mat-97-ExpandersLp}), for every $s,t \in \R$ and $\alpha \geq 1$,
\[ |\{s\}^\alpha-\{t\}^\alpha| \leq \alpha (|s|^{\alpha-1}+|t|^{\alpha-1}) |s-t|. \]
For each $z \in Z$ we apply this to $s=g(x)+C, t=g(y)+C, \alpha=\frac{q}{p}$ for all pairs of points $x,y \in B(z,a)$ and see that 
\[
	|\nabla_a f|(z) \leq \frac{2q}{p}\overline{(g+C)}_a(z)^{\frac{q-p}{p}}|\nabla_a g|(z).
\]
By the definition of $\nabla_a$, $\overline{(g+C)}_a(z)\leq \abs{g(z)+C} + \abs{\nabla_a g}(z)$, so taking $p$th powers and integrating, we see that
\begin{align*}
 & \norm{g+C}^q_q h^p_a(Z)^p 
   = \norm{f}_p^p h^p_a(Z)^p = \norm{f-f_Z}_p^p h^p_a(Z)^p
 \\
 &
 \leq \int_Z \abs{\nabla_a f}(z)^p d\nu 
 \\ &
 \leq \left(\frac{2q}{p}\right)^p \int_Z \left(\abs{g(z)+C}+\abs{\nabla_a g}(z)\right)^{q-p} |\nabla_a g|(z)^p d\nu
 \\
 &
 \overset{(\star)}{\leq} \left(\frac{2q}{p}\right)^p 2^{q-p} \left( \int_Z \abs{g(z)+C}^{q-p}\abs{\nabla_a g}(z)^p d\nu + \norm{\nabla_a g}_q^q \right)
 \\ &
 \overset{(\dagger)}{\leq} \frac{2^q q^p}{p^p} \left( \norm{g+C}_q^{q-p} \norm{\nabla_a g}_q^p+ 12^{q-p} \norm{g+C}_q^{q-p} \norm{\nabla_a g}_q^p \right)
 \\ &
 \preceq_{p,q} \norm{g+C}_q^{q-p} \norm{\nabla_a g}_q^p,
\end{align*}
where $(\star)$ follows from $(s+t)^\alpha \leq 2^\alpha(s^\alpha+t^\alpha)$ for any $s,t,\alpha>0$, and $(\dagger)$ follows from H\"older's inequality and $\norm{\nabla_a g}_q \leq 12\norm{g+C}_q$.
Rearranging, taking $p$th roots, and applying \eqref{eq:mean-quasi-min} we have
\[
\norm{g-g_Z}_q \leq 2 \norm{g+C}_q \preceq_{p,q} \frac{\norm{\nabla_a g}_q}{h_a^p(Z)}.
\qedhere
\]
\end{proof}

\begin{proof}[Proof of Proposition \ref{bprop:pqmonotone}] This is immediate from Proposition \ref{prop:monotonep}.
\end{proof}

\begin{remark} There are graphs $X$ of bounded degree containing expanders, and by Propositions \ref{prop:monotonep} and \ref{prop:Lambda1Sep}, $$\Lambda^p_X(r_n)\succeq_p\Lambda^1_X(r_n)\asymp \sep_X(r_n)\succeq r_n$$ on some unbounded subsequence $(r_n)$ \cite{HumSepExp}, but $\Lambda^\infty_X(r)\simeq r/\log(r)$ by Proposition \ref{bprop:P8growth}, so one should not expect universal constants (independent of $p,q$) in the above proposition.
\end{remark}

\section{Poincar\'e profiles of groups with polynomial growth}\label{sec:polygrowth}

The goal of this section is to prove the lower bound in Theorem \ref{bthm:CGLC-poly}.

Given a compactly generated locally compact group $G$, with compact symmetric generating set $K$, let $d=d_K$ be the associated word metric and let $\mu$ be a left-invariant Haar measure. We refer to the triple $(G,d,\mu)$ as a \textbf{metric measure CGLC group}. By Lemma~\ref{lem:PCbasics} and Corollary \ref{cor:ProfileIndependant}, the $L^p$-Poincar\'{e} profile of $G$ is well-defined (up to $\simeq$).

\begin{theorem}\label{thm:CGLC-poly-PI} 
Let $(G,d,\mu)$ be a metric measure CGLC group. If there exists some $m>0$ such that $\gamma(r)\asymp r^m$, then for every $p\in[1,\infty]$, $\Lambda_G^p(r)\gtrsim_p r^{\frac{m-1}{m}}$.
\end{theorem}

This theorem will be our goal for the section. Note that the $p=\infty$ case follows immediately from Proposition \ref{prop:Linfty_inverse growth}. Moreover, by Proposition \ref{prop:monotonep} $\Lambda_G^p\gtrsim_p \Lambda^1_G$ for all $p\in[1,\infty)$.
 Using Proposition \ref{prop:Lambda1Sep}, Theorem \ref{bthm:monotoneunderreg} and \cite[Proposition 2.4]{HumSepExp}, it will suffice to define a bounded degree graph $X$ large-scale equivalent to $G$ and prove that $\sep_X(r)\gtrsim r^{\frac{m-1}{m}}$, in the cut-set sense of \cite{BenSchTim-12-separation-graphs}.  Our method is to prove that a metric measure space version of the hypothesis of \cite[Proposition 2.4]{HumSepExp} holds for CGLC metric measure groups (Proposition \ref{prop:polygrowthcut}); the starting point for achieving this is a Poincar\'e inequality satisfied by $G$.

\subsection{A Poincar\'e inequality}
Poincar\'e inequalities are well known to hold for groups with polynomial growth, see for example~\cite{Saloff-Coste}.
In this subsection we present a generalisation of \cite[Theorem $2.2$]{Kleiner-10-Gromov-thm} (attributed to Saloff-Coste and explicitly appearing in the $L^2$ case in \cite{DiaLSC-Comparison}) to compactly generated locally compact groups in our framework. The proof below is also similar in nature to \cite[Proposition $11.17$]{HajKos00-Sobolev-met-Poincare} which is attributed to Varopoulos \cite{Var-harmonic-Lie}.

\begin{theorem}\label{thm:mmsPoincare}
Let $(G,d,\mu)$ be a metric measure CGLC group. Let $\Delta:G \ra \R$ be the modular function on $G$; i.e., for $U \subset G$ and $g \in G$, $\mu(Ug)=\mu(U)\Delta(g)$. Define $\Delta(K)=\sup_{g\in K} \Delta(g)$.

For any $p \geq 1$, $a\geq 1$, for any metric ball $B=B(x_0,R)$ of radius $R$ and any function $f\in L^p(G)$ we have the following:
\[
\int_{B} \abs{f(x)-f_B}^p d\mu(x) \leq \frac{(2R)^p\mu(2B)\Delta(K)^{2R}}{\mu(B)} \int_{3B} \abs{\nabla_a f}(x)^p d\mu(x),
\]
where for $\lambda>0$, $\lambda B = B(x_0,\lambda R)$.
\end{theorem}
\begin{proof}
	We may assume $x_0 = e$.
	Recall that 
	\[ \abs{\nabla_a f}(x)=\sup\setcon{\abs{f(y)-f(z)}}{y,z\in B(x,a)}. \]
	If $a\leq a'$ then $\abs{\nabla_a f}(x)\leq \abs{\nabla_{a'} f}(x)$ so it suffices to prove the result above for $a=1$.

For every $z \in 2B$, we choose a geodesic $\gamma_z:\set{0,1,\dots,k}\to G$ with $\gamma_z(0)=e$ and $\gamma_z(k)=z$.

For $x, y \in B(R)$, let $z = x^{-1}y$, and let $|\gamma_z|=k$ be the length of the corresponding path.
Then by the triangle and H\"older's inequality,
\begin{align*}
	|f(x)-f(y)|^p 
	\leq \left| \sum_{i=1}^{|\gamma_z|} |\nabla_{1} f|(x\gamma_z(i)) \right|^p
	 \leq |\gamma_z|^{p-1} \sum_{i=1}^{|\gamma_z|} |\nabla_1 f|(x\gamma_z(i))^p.
\end{align*}

For fixed $z \in 2B$, consider the map $F:(x,i) \mapsto (x\gamma_z(i),i)$.
This is clearly injective, so $(x,i) \mapsto x\gamma_z(i)$ is at most $2R$-to-1,
and
\begin{align*}
	\int_{B} \sum_{i=1}^{|\gamma_z|} |\nabla_{1} f|(x\gamma_z(i))^p d\mu(x)
	& = \sum_{i=1}^{|\gamma_z|} \int_B |\nabla_1 f|(x\gamma_z(i))^p d\mu(x)
	\\ & = \sum_{i=1}^{|\gamma_z|} \int_{B \cdot \gamma_z(i)} |\nabla_1 f|(x)^p \Delta(\gamma_z(i)^{-1}) d\mu(x) 
	\\ & \leq 2R \sup_{g \in 2B}\Delta(g) \int_{3B} |\nabla_1 f|(x)^p d\mu(x).
\end{align*}
Since $2B = K^{2R}$ we have $\sum_{g\in 2B} \Delta(g) = \Delta(K^{2R}) \leq \Delta(K)^{2R}$, so
\begin{align*}
	\int_{B} \abs{f-f_B}^p d\mu
	& \leq \int_B \left| \int_B |f(x_1)-f(x_2)| \frac{d\mu(x_2)}{\mu(B)} \right|^p d\mu(x_1) \\
	& \leq \frac{1}{\mu(B)} \int_{B\times B} \abs{f(x_1)-f(x_2)}^p d\mu(x_1)d\mu(x_2)
	\\ & \leq \frac{(2R)^{p-1}}{\mu(B)}
		\int_{x \in B} \int_{z \in 2B}
		\sum_{i=1}^{|\gamma_z|} |\nabla_{1} f|(x\gamma_z(i))^p d\mu(z) d\mu(x)
	\\ & \leq \frac{(2R)^p \Delta(K)^{2R}}{\mu(B)}
		\int_{z \in 2B} \int_{x \in 3B} |\nabla_1 f|(x)^p d\mu(x) d\mu(z)
	  \\ & \leq \frac{(2R)^p \Delta(K)^{2R} \mu(2B)}{\mu(B)}
		\int_{3B} |\nabla_1 f|(x)^p d\mu(x). \qedhere
\end{align*}
\end{proof}

\subsection{CGLC groups with polynomial growth}
We begin by refining the above Poincar\'e inequality.

\begin{lemma} If $\liminf_{r\to\infty} \frac{1}{r}\log(\mu(B(1,r)))=0$, then $G$ is unimodular.
\end{lemma}
\begin{proof} Suppose $G$ is not unimodular, then there exists some $g\in G$ such that $\Delta(g)>1$. Since $\Delta$ is multiplicative, there is some $k\in K$ with $\Delta(k)>1$.

  Now, for each $n$, $Kk^n\subseteq B(1,n+1)$, so $\mu(B(1,n+1))>\Delta(k)^n\mu(K)$, and therefore $\liminf_{r\to\infty} \frac{1}{r}\log(\mu(B(1,r)))>0$.
\end{proof}

From this we obtain the following refinement of a special case of Theorem \ref{thm:mmsPoincare}.

\begin{corollary}\label{cor:polygrowthPI} If $G$ has polynomial growth then there exists a constant $C$ such that, for any $p \geq 1$ and $a\geq 1$, for any metric ball $B=B(x_0,R)$ of radius $R$ and any function $f\in L^p(G)$ we have the following:
\begin{equation}\label{eq:cglcPolyPI}
	\int_{B} \abs{f(x)-f_B}^p d\mu(x) \leq C R^p \int_{3B} \abs{\nabla_a f}(x)^p d\mu(x).
\end{equation}
\end{corollary}

Using this refined Poincar\'{e} inequality (specifically the case $p=1$) we will now present a proof of Theorem \ref{thm:CGLC-poly-PI} via a series of lemmas. The goal is to prove that any subset $A$ of $B$ such that both $A\cap B$ and $A^c\cap B$ have measure proportional to $B$ must have large boundary inside $B$. It is not sufficient to apply the Poincar\'{e} inequality (\ref{eq:cglcPolyPI}) to the characteristic function of $A$ inside $B$ as we cannot distinguish the contribution coming from the boundary of $A$ in $B$ with that coming from the boundary of $B$ in $X$. The solution is to apply the Poincar\'{e} inequality (\ref{eq:cglcPolyPI}) ``deep inside'' $B$.

From this we will show that there is a large subset of $B$ with sufficiently large Cheeger constant. This step is modelled on ideas from \cite{HumSepExp} relating the cut size and Cheeger constant definitions of separation.

\begin{definition} Let $X$ be a metric space, let $x\in X$, and let $r,s\in\R_+$ with $s>r$. The $(r,s)$\textbf{-corona} around $x$ is the set $C_{r,s}(x) = B(x,s) \setminus B(x,r)$.
\end{definition}

\begin{lemma}\label{lem:smallcoronas}
Let $(G,d,\mu)$ be a metric measure CGLC group with $\gamma(r)\asymp r^m$. For each $\delta\in(0,1)$ there exists some $\eps>0$ such that for every $x\in G$ and $r$ sufficiently large, we have $\mu(C_{r,(1+\eps)r}(x)) \leq \delta \mu(B(x,r))$.
\end{lemma}
\begin{proof} Note that since $d$ is a word metric on $G$ with respect to a compact symmetric generating set $K$, $d$ is $1$-geodesic: for every pair of points $x,y\in G$ there is a sequence $x=x_0,\ldots,x_{d(x,y)}=y$ such that $d(x_0,x_i)=i$ for all $i$. Hence $(G,d)$ has Property (M) \cite[Proposition 2]{Tes-vol-spheres}. Therefore, by \cite[Lemma $24$]{Tes-vol-spheres}, there exist constants $\alpha,\beta>0$ independent of $r$ such that $\mu\left(C_{r-s,r}(x)\right) \geq \alpha \mu(C_{r,r+s}(x))$ for every $x\in G$ whenever $4\beta<s\leq r$. 

Let $\eps'\in(0,1)$, let $r\geq 8\beta$ and for each $1\leq i\leq k=\lfloor -\log_2\eps'\rfloor$, let $b_i=\mu(C_{(1-2^i\eps')r,r})$.

By construction $b_{i+1}\geq (1+\alpha)b_{i}$ for all $i\geq 1$, so $b_k\geq (1+\alpha)^{k-1} b_1$.

Fix $\delta\in(0,1)$. If $\mu(C_{r,(1+\eps')r})>\delta \mu(B(x,r))$, then $\mu(C_{r,(1+\eps')r})\geq \delta b_k \geq \delta (1+\alpha)^{k-1} b_1$. But, by \cite[Lemma $24$]{Tes-vol-spheres}, $b_1\geq \alpha \mu(C_{r,(1+\eps')r})$, so $\alpha\delta (1+\alpha)^{k-1}\leq 1$.

Thus $k\leq \log_{1+\alpha}(\frac{1}{\alpha\delta})+1$, which implies that 
$$\eps'\geq \eps_{\alpha,\delta}:=\frac{1}{4}(\alpha\delta)^{\log_{1+\alpha}(2)}$$ 
The conclusion of the lemma holds for all $\eps<\eps_{\alpha,\delta}$.
\end{proof}

\begin{lemma}\label{lem:deepsets} Let $(G,d,\mu)$ be a metric measure CGLC group with $\gamma(r)\asymp r^m$. There exist constants $r_0,\eps,k>0$ such that the following holds for all $r\geq r_0$.

For any $A\subset B(x,r)$ with $\frac14\gamma(r)\leq\mu(A)\leq\frac12\gamma(r)$, there exists a point $w\in B(x,r)$ such that $B(w,3\eps r)\subset B(x,r)$, and such that $\mu(B(w,\eps r) \cap A)\geq k \gamma(r)$ and $\mu(B(w,\eps r) \cap A^c)\geq k \gamma(r)$.
\end{lemma}

\begin{proof}  By Lemma \ref{lem:smallcoronas}, for all $r$ sufficiently large, and $\eps$ sufficiently small the corona $C_{(1-4\eps)r,r}(x)$  has size $<\frac{1}{10}\gamma(r)$ for every $x\in X$. 
Now fix $k>0$ such that $\gamma(\eps r)\geq \frac{80}{3}k\gamma(r)$ for all $r\geq r_0$. Applying Lemma \ref{lem:smallcoronas} with $\delta=\frac{k}{2}$ we deduce that
\begin{equation}\label{kbound} 
\gamma(\eps r + 1) -\gamma(\eps r) \leq \frac{k}{2} \gamma(\eps r) \leq \frac{k}{2}\gamma(r),
\end{equation}
holds whenever $\epsilon$ is sufficiently small and $r$ sufficiently large.

Since $\mu(A\cap B(x,(1-4\eps)r))\geq \frac{3}{20}\gamma(r)$, there exists a point $y\in B(x,(1-3\eps)r))$ such that $\mu(A\cap B(y,\eps r))\geq \frac{3}{20} \gamma(\eps r) \geq 2k\gamma(r)$. To see this, suppose for a contradiction that it is not the case. Then
\begin{align*}
  \frac{3}{20}\gamma(r)\gamma(\epsilon r)
  & \leq \int_{B(x,(1-4\epsilon)r)}\chi_A(z) \gamma(\epsilon r) d\mu(z)
  \\ & \leq \int_{B(x,(1-3\epsilon)r)} \int_{B(y,\epsilon r)} \chi_A(z) d\mu(z)\,d\mu(y)
  \\ & < \frac{3}{20}\gamma(\epsilon r) \gamma((1-3\epsilon)r) < \frac{3}{20}\gamma(\epsilon r)\gamma(r),
\end{align*}
a contradiction.

Similarly, there is some $z\in B(x,(1-3\eps)r))$ with $\mu(A^c\cap B(z,\eps r)) \geq 2k\gamma(r)$. Now, by our choice of $k$, for every $v\in B(x,(1-3\eps)r)$, $$\max\set{\mu(A\cap B(v,\eps r)),\mu(A^c\cap B(v,\eps r))} \geq \tfrac{1}{2}\gamma(\epsilon r)\geq 2k\gamma(r).$$

Since $y,z\in B(x,(1-3\eps)r))$ there is a sequence $y=v_0,v_1,\dots,v_l=z$ such that $d(v_{i-1},v_i)=1$, $l\leq 2r$ and $\set{v_i}\subset B(x,(1-3\eps)r)$. By \eqref{kbound}, we see that the measure of the symmetric difference of $B(v_i,\eps r)$ and $B(v_{i+1},\eps r)$ is at most $k\gamma(r)$ for all $i$.

Choose $i$ maximal such that $\mu(A\cap B(v_i,\eps r)) \geq 2k\gamma(r)$. 
If $i=l$ then we choose $w=v_l$ and the proof is complete.
If $i<l$ then $\mu(A^c\cap B(v_{i+1},\eps r)) \geq 2k\gamma(r)$, but since the symmetric difference of $B(v_i,\eps r)$ and $B(v_{i+1},\eps r)$ has measure at most $k\gamma(r)$, we see that $\mu(A^c\cap B(v_{i},\eps r)) \geq k\gamma(r)$ and we set $w=v_i$.
\end{proof}

With this lemma we can show that large subsets of balls have large boundaries inside the ball.

\begin{proposition}\label{prop:polygrowthcut} Let $(G,d,\mu)$ be a metric measure CGLC group with $\gamma(r)\asymp r^m$.  There exists a constant $k>0$ so that for every $a \geq 1$, for every ball $B$ of sufficiently large radius $r$, and any subspace $A\subset B$ with $\frac14\gamma(r)\leq\mu(A)\leq\frac12\gamma(r)$, we have $\mu(\partial_a^B A)\geq k r^{m-1}$.
\end{proposition}
\begin{proof}
Let $A\subset B(x,r)=B$ be such that $\frac14\mu(B)\leq\mu(A)\leq \frac12\mu(B)$. 
It suffices to prove the proposition for $a=1$, in which case $\partial_a^B A$ is the same whether computed with the $1$-geodesic metric on $B$ or with $d$.  
By Lemma \ref{lem:deepsets}, assuming $r$ is sufficiently large, there exists some $w\in B(x,(1-3\eps)r)$ such that $\mu(B(w,\eps r) \cap A),\ \mu(B(w,\eps r) \cap A^c)\geq k' \mu(A)$, where $k'>0$ is independent of $r$.  Applying the Poincar\'{e} inequality \eqref{eq:cglcPolyPI} with $p=1$ to the characteristic function $\mathbf{1}_{A}$ on the ball $B(w,\eps r)$ we see that
\[
	\tfrac12 k' \mu(A) \leq C\epsilon r \mu(\partial_1^{B(w,3\epsilon r)} A).
\]
Since $B(w,3\eps r)\subseteq B$ we deduce that there exists a constant $k>0$ (independent of $r$) such that
\[
 \mu(\partial_1^{B}A)\geq \frac{k}{r} \mu(B).\qedhere
\]
\end{proof}

The last step in the argument ensures that there is a large subset of the ball with suitable Cheeger constant at scale $a$ (compare with a similar result for graphs presented in \cite{HumSepExp}).

\begin{proof}[{Proof of Theorem \ref{thm:CGLC-poly-PI}}]
Let $(G,d,\mu)$ be a metric measure CGLC group with $\gamma(r)\asymp r^m$. 
Our goal is to show that for every $p \in [1,\infty]$ we have $\Lambda_G^p(\gamma(r)) \gtrsim_p \gamma(r)^{\frac{m-1}{m}}$.  As discussed above, by Propositions~\ref{prop:Linfty_inverse growth} and \ref{prop:monotonep} it suffices to show that $\Lambda_G^1(\gamma(r)) \gtrsim \gamma(r)^{\frac{m-1}{m}}$.

Let $Y$ be a maximal $3$-separated set in 
\[
	\setcon{y \in G}{\exists r \in \Z : d(1,y)=3r},
\]
then $Y$ is a $6$-net in $G$.  For each $r \geq 2$, for each $y \in Y$ satisfying $d(1,y)=3(r-1)$, set 
\[
B_y = \setcon{z \in B(1,3r)}{d_{B(1,3r)}(z,y) \leq 6},
\]
where $d_{B(1,3r)}$ is the $1$-geodesic metric on $B(1,3r)$.
 It follows that
\[
 B(y,3) \subseteq B_y \subseteq B(y,6) \quad \textup{and} \quad \bigcup_{y\in Y\cap B(1,3(r-1))} B_y = B(1,3r)
\]
holds for all $y\in Y$ and $r\geq 2$.

Let $X$ be the graph with vertex set $Y$ and edges between $y,y' \in Y$ whenever $d(y,y')\leq 13$; this graph has bounded degree.  With respect to the vertex counting measure on $X$, the natural inclusion $f:X\to G$ is a large-scale equivalence (cf.\ Definition \ref{def:largescaleEquiv}), so it suffices to prove lower bounds on the usual separation profile of $X$. Fix a large-scale equivalence $h: G\to X$ by sending each $g\in G$ to some $y\in Y$ with $g \in B_y$ so that $d(g,f(y))$ is minimal among all such $y$.  This map satisfies $B(y,1)\subset h^{-1}(y)\subset B(y,6)$ for each $y\in Y$.  Moreover, if $d(g,g')=1$ then $h(g)$ and $h(g')$ are adjacent or equal in $X$, so $h$ is $1$-Lipschitz.  Renormalise $\mu$ so that balls of radius $1$ have measure $1$ and then let $c$ be the measure of balls of radius $6$.

Let $\Gamma_r$ be the full subgraph of $X$ with vertex set $Y\cap B(1,3(r-1))$. 
Let $C\subset V\Gamma_r$ be such that any connected component of $V\Gamma_r\setminus C$ has at most $\delta \abs{\Gamma_r}$ vertices, for some constant $\delta$ to be chosen.

Suppose for a contradiction that $\abs{C}\leq \delta' r^{m-1}$. For $\delta,\delta'$ small enough, there is a union of connected components $D$ of $\Gamma_r\setminus C$ such that $A=h^{-1}(D)\subseteq B(1,3r)$ satisfies
\[
 \frac14\mu(B(1,3r)) \leq \mu(A) \leq \frac12\mu(B(1,3r)).
\]
This is possible, as $B:=B(1,3r)$ has measure at least $\abs{\Gamma_r}$ and for any connected component $A'$ of $V\Gamma_r\setminus C$ we have $h^{-1}(A') \subseteq \bigcup _{x\in A'} B(f(x),6)$ which has measure at most $c\abs{A'}\leq c\delta \abs{\Gamma_r}$. As long as $\delta\leq\frac{1}{2c}$ a simple greedy choice of connected components (ordered by the measure of their pre-images) yields the desired set $A$; we fix $\delta=\frac{1}{2c}$.

By Proposition \ref{prop:polygrowthcut}, there is a $k>0$ so that $\mu(\partial^B_1 A) \geq kr^{m-1}$.  
Recall that we consider $B$ as a subspace of $X$ equipped with its $1$-geodesic metric which we denote by $d_B$.

Now, by definition, if $z\in\partial^B_1 A$ then either $z\in A$ and there is some $z'\in B\setminus A$ with $d_B(z,z')\leq 1$ or the same distance bound holds with $z'\in A$ and $z\in B\setminus A$. We consider the first case; the second is similar. Since $B\setminus A \subseteq h^{-1}(V\Gamma_r\setminus D)$ there are vertices $x\in D$ and $x'\in V\Gamma_r\setminus D$ such that $d_B(z,f(x)), d_B(z',f(x'))\leq 6$, and so $d_B(f(x),f(x'))\leq 13$.  Since $h$ is $1$-Lipschitz, and $h(f(x))=x$, $d_{\Gamma_r}(x,x')\leq 13$. Now there is a path of length at most $13$ from $x$ to $x'$ in $\Gamma_r$ and at least one vertex in this path lies in $C$. Since adjacent vertices in $X$ are at most $13$ apart in $G$, $z$ is within a distance of $d(z,f(x))+13\cdot 13\leq 6+169=175$ of $f(C)$.
Therefore we have $\partial^B_aA\subseteq \bigcup_{x\in C} B_{175}(f(x))$, so
\[
 kr^{m-1} \leq \abs{C}\mu(B(1,175)),
\]
a contradiction for $\delta'$ sufficiently small, therefore any $C \subset V\Gamma_r$ so that all connected components of $V\Gamma_r\setminus C$ have size at most $\delta \abs{\Gamma_r}$ must have size $> \delta' r^{m-1}$.  In the terminology of \cite{BenSchTim-12-separation-graphs, HumSepExp}, $\cut^\delta(\Gamma_r) > \delta' r^{m-1}$.  So, by \cite[Proposition 2.4]{HumSepExp} and Remark~\ref{rmk:bdrydefs} there is a subgraph $\Gamma_r'$ of $\Gamma_r$ with $\abs{\Gamma_r'}\geq \frac12 \abs{\Gamma_r}$ and $r^m h_1(\Gamma_r') \succeq r^{m-1}$.
Therefore by Proposition~\ref{prop:Lambda1Sep}, and writing $n=\gamma(r)\asymp r^m$ for clarity,
\[
 \Lambda^1_G(n) \simeq \Lambda^1_X(n) \simeq \sep_X(n) \gtrsim n^{\frac{m-1}{m}}.
 \qedhere
\]
\end{proof}

\section{Upper bounds and large-scale dimension}\label{section:measureDimension}

	The goal of this section is to obtain upper bounds on the Poincar\'e profiles of a metric measure space which is finite dimensional in the sense of the definition below.  In doing so, we will prove that the lower bound for groups of polynomial growth in section \ref{sec:polygrowth} is sharp to complete the proof of Theorem~\ref{bthm:CGLC-poly}.

\begin{definition}\label{def:measurabledim} Let $(X,d,\mu)$ be a metric measure space. We say $X$ has \textbf{measurable dimension at most} $n$ ($\mdim(X)\leq n$) if, for all $r\geq 0$ we can write $X=X_0\cup \dots \cup X_n$ and decompose each $X_i=\bigcup X_{ij}$ so that each $X_{ij}$ is $1$-thick, $\sup(\mu(X_{ij}))<\infty$ and $d(X_{ij},X_{ij'})\geq r$ whenever $j\neq j'$.

If $\mdim(X)\leq n$ we define the function $\gamma_n(r)$ to be the infimal value of $\sup(\mu(X_{ij}))+1$ taken over all decompositions of $X$ satisfying the above hypotheses.
\end{definition}

Notice that $\gamma_n(r)$ is non-decreasing as a function of $r$.

A simple comparison can be made with asymptotic dimension when the metric measure space has \textbf{bounded geometry}: for all $r\geq 0$ there exists some $C_r$ such that $\mu(B(x,r))\leq C_r$ for all $x\in X$.

\begin{lemma} Let $(X,d,\mu)$ be a metric measure space with bounded geometry.
Then the asymptotic dimension of $X$ is at least $\mdim(X)$.
\end{lemma}
\begin{proof} Suppose $asdim(X)\leq n$. This implies that for all $r\geq 0$ one can decompose $X=X'_0\cup\ldots\cup X'_n$ and further decompose each $X'_i=\bigcup X'_{ij}$ so that $\sup\set{\diam(X'_{ij})}=K_{r}<\infty$ and $d(X_{ij},X_{ij'})\geq r+2$ whenever $j\neq j'$.

Define $X_{ij}=\bigcup_{y\in X'_{ij}} B(y,1)$. Each $X_{ij}$ is $1$-thick, it has diameter at most $L=K_{r}+2$ and $d(X_{ij},X_{ij'})\geq r$ whenever $j\neq j'$. Since $X$ has bounded geometry, $\mu(X_{ij})\leq C_L$ for all $i,j$.
\end{proof}

\begin{lemma} Let $(X,d,\mu)$ and $(Y,d',\mu')$ be metric measure spaces and suppose $Y$ has bounded packing at scales $\geq 1$. If there exists a coarse regular map $F:X\to Y$, then $\mdim(X)\leq\mdim(Y)$. Moreover, for all suitable $n$ we have $\gamma^X_{n}\lesssim_n \gamma^Y_{n}$.
\end{lemma}
\begin{proof} Suppose $\mdim(Y)\leq n$. Then for all $r\geq 0$ one can write $Y=\bigcup_{i=0}^n\bigcup_j Y^r_{ij}$ where each $Y^r_{ij}$ is $1$-thick, $\mu'(Y^r_{ij})\leq C$ for some $C$ and all $i,j$, and $d'(Y^r_{ij},Y^r_{ij'})>\rho_+(r+2)$ whenever $j\neq j'$.

Let $X^r_{ij}=[F^{-1}(Y^r_{ij})]_1$. By Definition \ref{def:coarseregular}(i), $d(X^r_{ij},X^r_{ij'})>r$ whenever $j\neq j'$, and by (ii) $\mu(X^r_{ij})\asymp \mu'([Y^r_{ij}]_1)\preceq \mu'(Y^r_{ij})$ by Lemma \ref{lem:doublingThick}.
\end{proof}

\begin{remark} One can remove the assumption that $Y$ has bounded packing at scales $\geq 1$ by removing the assumption that each $X_{ij}$ is $1$-thick in the definition of measurable dimension.
\end{remark}

\begin{proposition}\label{prop:mdim} Let $(X,d,\mu)$ be a metric measure space with $\mu(X)=\infty$ and measurable dimension at most $n$. For all $\delta>0$, 
$$\Lambda^p_{X}(r) \lesssim_{n} \sup\setcon{\gamma_{n}(t+\delta)/t}{\gamma_{n}(t)\leq r/(4n+4)}.$$
\end{proposition}

\begin{proof} If $\gamma_{n}$ is bounded then $\mu$ is bounded, which is a contradiction.

  Choose $s > 4(n+1)\gamma_n(0)$ and assume $\mu(A) = s\leq r$.
Fix $\delta>0$ and find $t$ so that $4(n+1)\gamma_{n}(t)\leq \mu(A) \leq 4(n+1)\gamma_{n}(t+\delta)$.
  Select a decomposition of $X$ into sets $X^t_{ij}$ as above where $\mu(X^t_{ij})\leq \gamma_{n}(t)$ for all $i,j$. 

Then there exists some $i$ such that $\mu(A\cap X_i)\geq \frac{1}{n+1}\mu(A) \geq 4\gamma_{n}(t)$. Without loss of generality, assume $i=0$. Choose $J$ so that 
\[X'_0:=\bigcup_{j\in J}X^t_{0j} \quad \textrm{satisfies} \quad \frac{\mu(A)}{4(n+1)}\leq \mu(A\cap X'_0)\leq \frac{\mu(A)}{2(n+1)}.\] 
Set $X''_0=X^t_0\setminus X'_0$ and let $f_t:A\to\R$ be the function $f(x)=\frac{1}{t}\min\set{t,d_X(x,X'_0)}$.

Now $f_t$ is $\frac{1}{t}$-Lipschitz, so $\int_A \abs{\nabla_2 f}^p \leq \frac{2^p}{t^p}\mu(A)$. Since $f$ takes values in $[0,1]$ and has value $0$ on $X'_0$ and value $1$ on $X''_0$ each of measure $\geq \mu(A)/4(n+1)$, we see that $\int_A \abs{f-f_R}^p d\mu(x) \geq (\frac{1}{2})^{p}\frac{1}{4(n+1)} \mu(A)$. 

Thus, $h^p_a(A)\leq \frac{4}{t}(n+1)^{2}\preceq_{n} \frac1t$. As this holds for every measurable $A\subset X$ of finite measure the result follows.
\end{proof}

\begin{remark} Under nice circumstances, for instance when a space $X$ has a cobounded isometry group, and finite asymptotic dimension where the $K_r$ can be bounded by an affine function of $r$ (sometimes called linearly controlled or asymptotic Assouad--Nagata dimension), the function $\gamma_{n}(s_r+\delta)/s_r$ is equivalent (up to $\simeq$) to $r/\kappa(r)$ where $\kappa$ is the inverse growth function. This is easily deduced from the argument in the proof of Proposition \ref{prop:Linfty_inverse growth}.
\end{remark}

\begin{proof}[Proof of Theorem $\ref{bthm:CGLC-poly}$]
	Let $(G,d,\mu)$ be a CGLC metric measure group with $\mu(B(1,r))\asymp r^m$. Such groups have finite asymptotic Assouad--Nagata dimension (\cite[Theorem 1.2]{Breu-14-geom-loc-cpt-poly-growth} and \cite[Theorem 5.5]{HP-13-ANdimension-nilppolyc}), so by Proposition \ref{prop:mdim}, $\Lambda^p_G(r)\lesssim r^{\frac{m-1}{m}}$ for all $p\geq 1$. The lower bound is proved in Theorem \ref{thm:CGLC-poly-PI}.
\end{proof}

\begin{example}
As another example, for $X$ equal to the product of two $3$-regular trees we have $\Lambda^p_X(r)\simeq r/\log(r)$ for all $p\in[1,\infty]$: The case $p=\infty$ follows immediately from Proposition \ref{prop:Linfty_inverse growth}. By \cite[Theorem $3.1$]{BenSchTim-12-separation-graphs} and Proposition \ref{prop:Lambda1Sep}, the lower bound holds when $p=1$, so the lower bound for general $p$ follows from Proposition \ref{prop:monotonep}. For the upper bound, $X$ has exponential growth, a cobounded isometry group, and asymptotic Assouad--Nagata dimension $2$, so by Proposition \ref{prop:mdim}, $\Lambda^p_X(r)\lesssim r/\log(r)$ for all $p\geq 1$.
\end{example}

\section{Trees}\label{sec:trees}

In this section, we calculate the Poincar\'e profile for regular trees.
\begin{theorem}[Theorem~\ref{bthm:trees}]\label{thm:tree-profile}
  Let $T$ be the infinite $3$-regular tree.
  Then for every $p\in [1,\infty)$, $\Lambda^p_T(r) \asymp_p r^{(p-1)/p}$.
\end{theorem}
For $p=1$ this is immediate from \cite{BenSchTim-12-separation-graphs}. This theorem immediately implies the following corollary for groups admitting quasi-isometric embeddings of such trees.
\begin{corollary}\label{cor:hyp-lower-bound}
  If $(G,d,\mu)$ is a CGLC measure group which is non-amenable, non-unimodular, or is compact-by-elementary amenable and has exponential growth, then for any $p \geq 1$, $\Lambda^p_G(r) \gtrsim_{G,p} r^{(p-1)/p}$.
\end{corollary}
\begin{proof}
In the first two cases this follows from \cite{Ben-Sch-trees-in-nonamen}, and in the third from \cite{Chou}.
\end{proof}

In this section, for a graph $X$, and a function $f:VX \ra \R$, we define $|\nabla f|: EX \ra \R$ as $|\nabla f|(e) = |f(x)-f(y)|$ where $e\in EX$ has endpoints $x,y \in VX$.  If $X$ has maximum vertex degree $d$ then for each $p \geq 1$,
\[
  \| \nabla_2 f \|_p \asymp_d \| \nabla f \|_p = \left(\sum_{e \in EX} |\nabla f|(e)^p \right)^{1/p}.
\]
A key step in proving Theorem~\ref{thm:tree-profile} is to reduce to an estimate on complete graphs in the spirit of, for example, Spielman~\cite[Section 4.7]{Spielman-course15}.
\begin{proposition}\label{prop:complete-graph-cheeger}
  For any $r \in \N, r \geq 2$ and $p \in [1,\infty)$, letting $K_r$ denote the complete graph on $r$ vertices, we have 
  \[ 
  r^{1/p} \leq \inf \setcon{\frac{\| \nabla f\|_p}{\|f - f_{K_r}\|_p}}{f:VK_r \ra \R, f \not\equiv f_{K_r}}\preceq_p r^{1/p}.
  \]
\end{proposition}
\begin{proof}
  Let $f : VK_r \ra \R$ be any non-constant function on $K_r$.  Then
  \begin{align*}
    \| f-f_{K_r} \|_p^p
    & = \sum_x \left| f(x) - \frac{1}{r} \sum_y f(y) \right|^p
    \\ & \leq \frac{1}{r^p} \sum_x \left( \sum_y | f(x) - f(y) | \right)^p
    \\ & \leq \frac{1}{r^p} \sum_x \left( \sum_y | f(x) - f(y) |^p \right) r^{p-1}
    \\ & = r^{-1} \| \nabla f \|_p^p.
  \end{align*}
  This proves the first inequality; the second can be seen by considering a function which is $1$ and $-1$ on one vertex each, and zero everywhere else.
\end{proof}

\begin{proof}[Proof of Theorem~\ref{thm:tree-profile}]
  First we show the upper bound, which is relatively simple.

  Suppose $A \subset T$ is a graph of size $|A|=r$; we can find a vertex $x$ so that on deleting this vertex, all remaining connected components have size $\leq r/2$.  Group these components into sets $U, V$ of size $\in [r/4, 3r/4]$.  Let $f:A \ra [-1,1]$ be identically $-1$ on $U$, $1$ on $V$ and $0$ on $x$.  

  Clearly $\| f-f_A \|_p^p \geq \frac14 r$, and since $\nabla f$ is only non-zero on edges adjacent to $x$, $\| \nabla f \|_p^p \leq 3$.  Thus $h^p(A) \leq (12/r)^{1/p}$ and \[ \Lambda^p(r) = \sup_{|A|\leq r} |A|h^p(A) \leq 12r^{(p-1)/p}.\]

  Second, we show the lower bound.

  For any $r > 0$ there exists a ball $B=B(x_0,t) \subset T$ of size $\asymp 2^t \asymp r$, so we can assume $r = |B|$ and it then suffices to show that $h^p( B ) \succeq |B|^{-1/p}$, with constant independent of $B$.

  Let $K_r$ be the complete graph on $r$ vertices.
  Suppose that a non-constant function $f : B \ra \R$ is given.  Consider $f$ as a function on the complete graph $K_r = K_{|B|}$.  In light of Proposition~\ref{prop:complete-graph-cheeger}, to show that $h^p(B) \succeq |B|^{-1/p}$, it suffices to show that
  \[
    \sum_{e \in EB} |\nabla f(e)|^p \geq \frac{1}{2|B|^2} \sum_{x,y \in B} |f(x)-f(y)|^p,
  \]
  for then $h^p(B) \succeq |B|^{-2/p} r^{1/p} \succeq |B|^{-1/p}$.

  Now for each $x,y \in B$, let $\gamma_{xy}$ be the simple path in $T$ joining $x$ to $y$.  Observe that $|f(x)-f(y)| \leq \sum_{e \in \gamma_{xy}} |\nabla f(e)|$.
  
  For each $e \in EB$, let $N_e$ be the number of such simple paths that pass through $e$.  Simple paths passing through $e$ are in one-to-one correspondance with pairs $(v,w)$, where $v,w \in B$ are in different components of $B$ with $e$ deleted.  The component containing $x_0$ has size $\asymp 2^t \asymp |B|$, while the component not containing $x_0$ has size $\asymp 2^{t-d(x_0,e)}$ where $d(x_0,e)$ is the distance from the centre of the ball to the edge $e$.  So we deduce that $N_e \asymp 2^t \cdot 2^{t-d(x_0,e)}$.  Using H\"older's inequality, we have
\begin{align*}
  \sum_{x,y \in B} |f(x)-f(y)|^p
  & \leq \sum_{x,y \in B} \left( \sum_{e \subset \gamma_{xy}} |\nabla f(e)| \right)^p \\
  & = 
  \sum_{x,y \in B} \left( \sum_{e \subset \gamma_{xy}} |\nabla f(e)| N_e^{-1/p} N_e^{1/p} \right)^p \\
  & \leq \sum_{x,y \in B} \left( \sum_{e \subset \gamma_{xy}} |\nabla f(e)|^p N_e^{-1} \right)  
  \left( \sum_{e \subset \gamma_{xy}} N_e^{1/(p-1)} \right)^{p-1}
\end{align*}
For each simple path, $N_e^{1/(p-1)}$ takes values in (two) geometric series, with ratio depending only on $p$ and maximum value $\preceq (2^{2t})^{1/(p-1)} \asymp |B|^{2/(p-1)}$, and so the sum inside the second parentheses above is also $\preceq |B|^{2/(p-1)}$.
Thus,
\begin{align*}
  \sum_{x,y \in B} |f(x)-f(y)|^p 
  & \preceq 
  \sum_{x,y \in B} \sum_{e \in \gamma_{xy}} |\nabla f(e)|^p N_e^{-1} |B|^2
  \\ & \leq 2 |B|^2 \sum_{e \in B} |\nabla f(e)|^p,
\end{align*}
and so we are done.
\end{proof}

\section{Lower bounds for hyperbolic spaces with boundary Poincar\'{e} inequalities}
  \label{sec:hyp-PI}

  In this section we find lower bounds on Poincar\'e profiles for hyperbolic groups whose boundaries admit Poincar\'e inequalities in the style of Heinonen and Koskela (Theorem~\ref{bthm:lbdQrgPI}). In section~\ref{sec:applications} we will apply these results to rank $1$ symmetric spaces, and a family of hyperbolic buildings studied by Bourdon and Pajot.

Suppose a metric space $(Z,\rho)$ is \textbf{Ahlfors $Q$-regular}, i.e.\ there is a measure $\mu$ on $Z$ so that for every ball $B(z,r)$ in $Z$ with $r \leq \diam(Z)$, we have $\mu(B(z,r)) \asymp r^Q$.  (We may take $\mu$ to be the Hausdorff $Q$-measure on $Z$.)
For $p, q \geq 1$, we say $(Z,\rho)$ admits a \textbf{$(q,p)$-Poincar\'e inequality} (with constant $L \geq 1$) if for every Lipschitz function $f:Z \ra \R$ and every ball $B(z,r) \subset Z$,
\begin{equation*}
  \left( \dashint_{B(z,r)} |f-f_{B(z,r)}|^q \, d\mu \right)^{1/q} \leq Lr \left( \dashint_{B(z,Lr)} (\Lip_x f)^p \, d\mu(x) \right)^{1/p} \ ,
\end{equation*}
where for $U \subset Z$, $f_U = \dashint_U f\, d\mu= \frac{1}{\mu(U)} \int_U f\, d\mu$,
and \[ \Lip_x f = \limsup_{r\ra 0} \sup_{y \in B(x,r)} \frac{|f(y)-f(x)|}{r} .\]
If $q=1$, we say $Z$ admits a \textbf{$p$-Poincar\'e inequality}.
Note that the Poincar\'e inequality above is a variation of Heinonen and Koskela's original that is shown to be equivalent by Keith~\cite[Theorem 2]{Kei-03-mod-pi}.

By H\"older's inequality, if $Z$ admits a $p$-Poincar\'e inequality, it admits a $q$-Poincar\'e inequality for all $q \geq p$.  Moreover, since $Z$ is doubling, it will admit $(q,q)$-Poincar\'e inequalities for all $q \geq p$ by \cite[Theorem~5.1]{HajKos00-Sobolev-met-Poincare}.

\smallskip
A geodesic metric measure space $(X,d,\mu)$ is \textbf{Gromov hyperbolic} if it is $\delta$-hyperbolic for some $\delta\geq 0$: for every geodesic triangle $T=(\gamma_1,\gamma_2,\gamma_3)$, we have $\gamma_1\subseteq [\gamma_2\cup\gamma_3]_\delta$.
It is \textbf{visual} if there exists $x_0 \in X$ and $C \geq 0$ so that every $x \in X$ belongs to a $C$-quasi-geodesic ray $\gamma:[0,\infty)\ra X$ with $\gamma(0)=x_0$. 
Gromov hyperbolic metric spaces have a boundary at infinity $\bdry X$ which comes with a family of metrics:
if $X$ is visual with respect to $x_0$, a \textbf{visual metric} $\rho$ on $\bdry X$ based at $x_0\in X$ with visibility parameter $\eps>0$ is a metric satisfying $\rho(\cdot,\cdot)\asymp \exp(-\eps(\cdot|\cdot)_{x_0})$, where $(\cdot|\cdot)_{x_0}$ denotes the Gromov product with respect to $x_0$.
For more background and discussion, see~\cite{BS-00-gro-hyp-embed, BP-03-lp-besov}.

We can now state the first main result of this section (cf.\ Theorem~\ref{bthm:lbdQrgPI}).
\begin{theorem}\label{thm:hyp-PI-bdry}
	Suppose that $X$ is a visual Gromov hyperbolic graph with a visual metric $\rho$ on $\bdry X$ that is Ahlfors $Q$-regular and admits a $p$-Poincar\'e inequality.  Then for all $q\geq p$, $\Lambda_X^q(r)\gtrsim r^{1-1/Q}$. 
\end{theorem}
By taking discretizations, one can apply this result to rank-$1$ symmetric spaces, amongst other examples.

\begin{proof}
Consider $\bdry X$ with the metric $\rho$, which admits a $p$-Poincar\'e inequality with some constant $L \geq 1$.  
As a consequence, $(\bdry X, \rho)$ is quasi-convex\footnote{This result is usually attributed to Semmes, a full proof can be found in \cite[\S17]{Che-99-diff}.}, so $\rho$ is bi-Lipschitz equivalent to a geodesic metric. Therefore we may assume that $\rho$ is geodesic, and so our standing assumptions hold.

Following Bourdon--Pajot~\cite[Section 2.1]{BP-03-lp-besov}, we ensure that $Z = (\bdry X, \rho)$ has diameter $1/2$ by rescaling, and define a graph $\Gamma$ which approximates $Z$: $\Gamma$ has vertex set $\{z_t^i : t \in \N, 1 \leq i \leq k(t)\}$ where for each $t \in \N$, $\Gamma_t=\{z_t^1,\ldots,z_t^{k(t)}\}$ is a maximal $e^{-t}$-separated net in $Z$.  To each $z_t^i$ we associate a ball $B(z_t^i,e^{-t}) \subset Z$, and we join $z_t^i$ and $z_u^j$ by an edge if and only if $|t-u| \leq 1$ and $B(z_t^i,e^{-t}) \cap B(z_u^j, e^{-u}) \neq \emptyset$.  

By Bourdon--Pajot~\cite[Proposition 2.1, Corollary 2.4]{BP-03-lp-besov}, $\Gamma$, with the path metric $d$, is a bounded degree hyperbolic graph which is quasi-isometric to $X$, and so it suffices to show the Poincar\'e profile bound for $\Gamma$.  

We now consider the sequence $Z_t=(Z, \rho_t, \mu_t)$ of metric measure spaces, where $\rho_t=6e^t\rho$, and $\mu_t=e^{Qt}\mu$. Note that $\mu_t(Z_t)\asymp e^{Qt}$. We deduce from the Poincar\'e inequality satisfied by $Z$ that $Z_t$ satisfies for any Lipschitz function $f$ on $Z_t$, for all $q\geq p$
$$ \left(\dashint |f-f_{Z_t}|^q \, d\mu_t\right)^{1/q} \preceq  e^{t}\left(\dashint(\Lip_x f)^q \, d\mu_t(x)\right)^{1/q},$$
and therefore that
$$h_{\Lip}^q(Z_t)\succeq e^{-t}$$
with constant independent of $t$.
By Proposition \ref{prop:UpperGradToLargeScale}, this implies that 
$$h_{2}^q(Z_t)\succeq e^{-t}.$$
Now equip $\Gamma_t$ with the counting measure and the distance induced from its inclusion in $Z_t$. Since $\Gamma_t$ is a maximal $6$-separated subset of $Z_t$, we can find a measurable partition 
$$Z_t=\bigsqcup_{\gamma\in \Gamma_t} A_{\gamma},$$ 
where  
$$B_{\rho_t}(\gamma,2)\subset A_{\gamma}\subset B_{\rho_t}(\gamma,18).$$ 
By the Ahlfors regularity of $Z_t$, $\mu(A_{\gamma})\asymp 1$. Hence by Lemmas \ref{lem:disc} and \ref{lem:PCbasics}\ref{PCchangemeasure}, we deduce that
 $$h_{40}^q(\Gamma_t,\rho_t)\succeq e^{-t}.$$
In order to conclude, we need to show that there exists a constant $C$ such that for every $t$ and every pair of vertices  $x,y\in\Gamma_t$ such that  $\rho_t(x,y)\leq 40$ satisfy  $d(x,y)\leq C$ (where $d(x,y)$ is their distance in $\Gamma$). Indeed, that will show that 
\begin{equation} \label{eq:hyp-gammat-h}
  h^q_C(\Gamma_t, d) \succeq e^{-t},
\end{equation}
and since $|\Gamma_t|\asymp e^{Qt}$,
$$\Lambda_{\Gamma}^q(r)\succeq r^{1-1/Q}.$$

By \cite[Lemma 2.2]{BP-03-lp-besov}, for $x,y \in \Gamma$ corresponding to balls $B_x, B_y \subset Z$, $e^{-(x|y)} \asymp \diam(B_x \cup B_y)$, where $(x|y)$ denotes the Gromov product with respect to the base point $z_1^1$.  For $x,y \in \Gamma_t$, we have $(x|y)$ equal to $t-\frac12 d(x,y)$ up to a uniform additive error, and $\diam(B_x \cup B_y) \asymp e^{-t} + \rho(x,y)$, so
\[
	e^{-t} e^{\tfrac12 d(x,y)} \asymp \diam(B_x \cup B_y) \asymp e^{-t} + \rho(x,y).
\]

Thus, $\rho(x,y)\leq \frac{40}{6}e^{-t}$ implies that $d(x,y)\preceq 1$, which completes the proof of Theorem~\ref{thm:hyp-PI-bdry}.
\end{proof}

We will see in section~\ref{sec:applications} that for many spaces, Theorem~\ref{thm:hyp-PI-bdry} gives sharp lower bounds for $\Lambda^q_X$ when $q \in [1,Q)$.
For $q=Q$, however, one can do better.
\begin{theorem}
  \label{thm:hyp-PI-bdry-sharp}
  Suppose that $X$ is a visual Gromov hyperbolic graph with a visual metric $\rho$ on $\bdry X$ that is Ahlfors $Q$-regular and admits a $Q$-Poincar\'e inequality.  Then $\Lambda_X^Q(r)\gtrsim r^{1-1/Q}\log(r)^{1/Q}$. 
\end{theorem}
\begin{proof}
We continue with the notation of the proof of Theorem~\ref{thm:hyp-PI-bdry}.
Given $s < t \in \N$, let $B_{s,t}$ be the full subgraph of $\Gamma$ containing the layers $\Gamma_{s+1}, \Gamma_{s+2},\ldots, \Gamma_t$.  (Later we will take $s = \lfloor t/2 \rfloor$.) The strategy of the proof is to use the Poincar\'e inequality in each layer to get a stronger constant for all of $B_{s,t}$.

Let us be given a function $f:B_{s,t} \ra \R$, i.e.\ a function on $VB_{s,t}$.
For $x\in \Gamma$, define $i_x \in \N$ to satisfy $x \in \Gamma_{i_x}$.
Given $x \in \Gamma$ and $i \leq i_x$, let $\pi_i(x) \in \Gamma_i$ be (one of) the points in $\Gamma_i$ so that the point in $Z$ corresponding to $x$ lies in the ball of radius $e^{-i}$ corresponding to $\pi_i(x)$; the allowed choices of $\pi_i(x)$ are all at distance $1$ from each other.

For $i = s+1, \ldots, t$, there are $i-s$ layers in $B_{s,t}$ with labels $\leq i$.  
\begin{lemma}\label{lem:colour} There is an assignment $B_{s,t} \ra \N$ that maps each $x \in B_{s,t}$ to a layer $c_x \in \{s+1,\ldots,i_x\}$,  so that for any $z \in B_{s,t}$ and any $c, i$ with $c \leq i_z \leq i \leq t$ we have 
\begin{equation}\label{eq:hypsharp0}
	|\{ x \in \Gamma_i : \pi_{i_z}(x)=z \text{ and } c_x = c \}| \preceq \frac{e^{Q(i-i_z)}}{i-s} \leq \frac{e^{Q(t-s)}}{t-s},
\end{equation}
where the constant of `\,$\preceq$\!' is independent of $s, t, z, c$ and $i$.
\end{lemma}
This follows from a colouring argument that we defer until later.

Similarly to the proofs in Section~\ref{sec:trees}, we bound
\begin{align}
  \| f-f_{B_{s,t}} \|_p^p
  & = \sum_x \left| f(x) - \frac{1}{|B_{s,t}|} \sum_y f(y) \right|^p
  \notag
  \\ & \leq \frac{1}{|B_{s,t}|} \sum_x \sum_y | f(x) - f(y) |^p .
  \label{eq:hypsharp1}
\end{align}
(We refrain from setting $p=Q$ at present to clarify the role this plays in the proof.)

At a cost of multiplying by $2$, we can restrict to sum only over $x,y$ where $i_x \leq i_y$.  In particular, $c_x \leq i_y$.
Given such $x,y$, we consider the path $\alpha_x$ that follows $x, \pi_{i_x-1}(x), \ldots, \pi_{c_x}(x)$, and also the path $\beta_{x,y}$ that follows along $\pi_{c_x}(y), \pi_{c_x+1}(y), \ldots, \pi_{i_y-1}(y), y$. 

Continuing from \eqref{eq:hypsharp1}, since
\begin{multline*}
	|f(x)-f(y)| \leq \Bigg( \sum_{z\in \alpha_x} |\nabla_1 f|(z) \Bigg)
  +  |f(\pi_{c_x}(x)) - f(\pi_{c_x}(y)) | \\ + 
   \Bigg( \sum_{z \in \beta_{x,y}} |\nabla_1 f|(z) \Bigg),
\end{multline*}
we use the inequality $(a+b+c)^p \leq 3^p (a^p+b^p+c^p)$ 
to find the following:
\begin{multline}\label{eq:hypsharp2}
  \| f-f_{B_{s,t}} \|_p^p
   \preceq 
   \frac{1}{|B_{s,t}|} \sum_{\substack{x,y \\ i_x \leq i_y} }
   \left( \sum_{z\in \alpha_x} |\nabla_1 f|(z)\right)^p 
	+ \\ 
	\frac{1}{|B_{s,t}|} \sum_{\substack{x,y \\ i_x \leq i_y} }
	|f(\pi_{c_x}(x)) - f(\pi_{c_x}(y)) |^p
	+ 
   \frac{1}{|B_{s,t}|} \sum_{\substack{x,y \\ i_x \leq i_y} }
   \left( \sum_{z \in \beta_{x,y}} |\nabla_1 f|(z) \right)^p .
\end{multline}
We denote the resulting three terms of the sum by $S_1$, $S_2$, and $S_3$.
For each $z \in B_{s,t}$, let $M_z$ be the number of pairs $(x,y)$ so that $\alpha_x$ passes through $z$, and likewise $N_z$ for $\beta_{x,y}$.
Let us bound the first term of \eqref{eq:hypsharp2}, $S_1$.
\begin{align*}
	S_1 & = \frac{1}{|B_{s,t}|} \sum_{\substack{x,y \\ i_x \leq i_y} } \left( \sum_{z \in \alpha_x} |\nabla_1 f|(z) M_z^{-1/p}M_z^{1/p} \right)^p
	\\ & \leq \frac{1}{|B_{s,t}|} \sum_{\substack{x,y \\ i_x \leq i_y} } \left( \sum_{z \in \alpha_x} |\nabla_1 f|(z)^p M_z^{-1}\right)
	\left( \sum_{z \in \alpha_x} M_z^{1/(p-1)} \right)^{p-1}
\end{align*}
when $p >1$.
If $z \in \Gamma_{s+j}$ for some $j \in \{1, \ldots, t-s\}$,
then by \eqref{eq:hypsharp0} the number of possible choices of $x$
is
\[
	\preceq \sum_{i=s+j}^t j \frac{e^{Q(i-s-j)}}{i-s} 
	\preceq \frac{j}{t-s} e^{Q(t-s-j)}
\]
and there are $\leq |B_{s,t}|$ possible choices of $y$ so that $z \in \alpha_x$.  Thus 
\[
 M_z \preceq \frac{j}{t-s} e^{Q(t-s-j)} |B_{s,t}| = 
  \frac{j}{e^{Qj}} \cdot \frac{e^{Q(t-s)}|B_{s,t}|}{t-s}.
\]
For any $p >1$, $\sum_{j \geq 1} (je^{-Qj})^{1/(p-1)}$ is bounded by some constant depending only on $Q$ and $p$.  Whether $p>1$ or $p=1$, we get that
\begin{align*}
	S_1 & \preceq \frac{e^{Q(t-s)}}{t-s} \sum_{\substack{x,y \\ i_x \leq i_y} } \left( \sum_{z \in \alpha_x} |\nabla_1 f|(z)^p M_z^{-1}\right)
	\\ & = \frac{e^{Q(t-s)}}{t-s} \sum_z |\nabla_1 f|(z)^p \left( \sum_{\substack{x,y : i_x \leq i_y, z \in \alpha_x}} M_z^{-1} \right)
	 = \frac{e^{Q(t-s)}}{t-s}  \left\| \nabla_1 f \right\|_p^p.
\end{align*}

A very similar calculation lets us bound $S_3$:
if $z \in \Gamma_{s+j}$ for some $j \in \{1,\ldots,t-s\}$, then by \eqref{eq:hypsharp0} there are $\preceq \frac{j}{t-s}|B_{s,t}|$ possible choices of $x$ and $\preceq e^{Q(t-s-j)}$ possible choices of $y$ so that $z \in \beta_{x,y}$.
Thus
\[
	N_z \preceq \frac{j}{t-s} |B_{s,t}| \cdot e^{Q(t-s-j)} = \frac{j}{e^{Q_j}} \cdot \frac{e^{Q(t-s)}|B_{s,t}|}{t-s},
\]
and the rest of the calculation goes through as before to give
$S_3 \preceq \frac{1}{t-s}e^{Q(t-s)} \| \nabla_1 f \|_p^p$.

It remains to bound $S_2$.
Suppose we have $x', y' \in \Gamma_{s+j}$ for some $j \in \{1,\ldots,t-s\}$.
Let $P_{x',y'}$ be the number of pairs $x,y \in B_{s,t}$ so that $i_x \leq i_y$ and $\pi_{c_x}(x)=x'$ and $\pi_{c_x}(y)=y'$.
Using again \eqref{eq:hypsharp0}, we can bound $P_{x',y'}$ by the product of the number of choices of $x$, which is $\preceq \frac{1}{t-s} e^{Q(t-s-j)}$, and the number of choices of $y$, which is $\preceq e^{Q(t-s-j)}$.
Thus
\begin{align}
	S_2 & = \frac{1}{|B_{s,t}|} \sum_{\substack{x,y \\ i_x \leq i_y} }
	|f(\pi_{c_x}(x)) - f(\pi_{c_x}(y)) |^p
	\notag \\ & = \frac{1}{|B_{s,t}|} \sum_{j=1}^{t-s} \sum_{x',y' \in \Gamma_{s+j}} P_{x',y'} |f(x')-f(y')|^p
	\notag \\ & \preceq \frac{e^{2Q(t-s)}}{(t-s)|B_{s,t}|} \sum_{j=1}^{t-s} e^{-2Qj} \sum_{x',y' \in \Gamma_{s+j}} |f(x')-f(y')|^p . \label{eq:hypsharp3}
\end{align}
Fixing for a moment our choice of $j$, let $f_j$ be the average value of $f$ restricted to $\Gamma_{s+j}$.
Assuming $Z=(\bdry X, \rho)$ satisfies a $p$-Poincar\'e inequality,
we apply \eqref{eq:hyp-gammat-h} to $\Gamma_{s+j}$ to obtain: 
\[
 \sum_{x' \in \Gamma_{s+j}} \left| f(x') - f_j \right|^p \preceq 
 e^{p(s+j)} \sum_{x' \in \Gamma_{s+j}} |\nabla_C f|(x')^p.
\]
Applying this twice, we have that
\begin{align*}
	\sum_{x',y' \in \Gamma_{s+j}} |f(x')-f(y')|^p 
	& \leq 2^p 
	\sum_{x',y' \in \Gamma_{s+j}} \left( |f(x')-f_j|^p +|f(y')-f_j|^p \right)
	\\ & \preceq e^{p(s+j)} |\Gamma_{s+j}| \sum_{x' \in \Gamma_{s+j}} |\nabla_C f|(x')^p.
\end{align*}
Since $|\Gamma_{s+j}|\asymp e^{Q(s+j)}$, and $|B_{s,t}| \asymp e^{Qt}$, on
substituting this back in to \eqref{eq:hypsharp3}, we get
\begin{align*}
 	S_2 & \preceq
	\frac{e^{2Q(t-s)}}{(t-s)|B_{s,t}|} \sum_{j=1}^{t-s} e^{-2Qj}e^{p(s+j)} e^{Q(s+j)} \sum_{x' \in \Gamma_{s+j}} |\nabla_C f|(x')^p
	\\ & \asymp \frac{e^{ Qt+(p-Q)s}}{t-s} \sum_{j=1}^{t-s} e^{(p-Q)j} \sum_{x'\in\Gamma_{s+j}} |\nabla_C f|(x')^p.
\end{align*}
Provided $p=Q$, this simplifies to 
\[
	S_2 \preceq \frac{e^{Qt}}{t-s} \left\| \nabla_C f \right\|_Q^Q.
\]

Our bounds for $S_1$ and $S_3$ are dominated by our bounds for $S_2$, so we set $s = \lfloor t/2 \rfloor$ and conclude by \eqref{eq:hypsharp2} that
\[
\|f-f_{B_{s,t}} \|_Q^Q \preceq 
	\frac{e^{Qt}}{t} \left\| \nabla_C f \right\|_Q^Q 
	\asymp \frac{|B_{s,t}|}{\log |B_{s,t}|} \left\| \nabla_C f \right\|_Q^Q . \qedhere
\]
\end{proof}
It remains to show the colouring argument giving~\eqref{eq:hypsharp0}.
\begin{proof}[Proof of Lemma~\ref{lem:colour}]
	Recall that we are defining a colouring $B_{s,t}\ra \{s+1,\ldots,t\}, x \mapsto c_x$.

	For each $i \in \{s+1,\ldots,t\}$, the vertices of $\Gamma_i$ correspond to a maximal $e^{-i}$-separated net in $Z$.
	By Ahlfors $Q$-regularity, there exists $C$ so that the number of $e^{-i}$ separated points in any $r$-ball in $Z$ is $\leq C (r/e^{-i})^Q = C r^Q e^{iQ}$.
	So if we let $r_i = \tfrac12 (i-s)^{1/Q} C^{-1/Q} e^{-i}$, we guarantee that any $r_i$-ball in $Z$ meets at most $(i-s)$ points corresponding to vertices of $\Gamma_i$.

	Define $\Gamma_i \ra \{s+1,\ldots,i\}, x \mapsto c_x$ to be any mapping so that no two points at distance $\leq r_i$ in $Z$ are mapped to the same value.  The existence of such a mapping follows from Zorn's lemma applied to the collection of all such partially defined functions.

	Doing this for each $i$, we obtain our mapping $B_{s,t} \ra \{s+1,\ldots,t\}$.  To verify that \eqref{eq:hypsharp0} holds, observe that for any $z \in B_{s,t}$ and $c, i$ satisfying $c \leq i_z \leq i \leq t$ the set $\{x \in \Gamma_i : \pi_{i_x}(x)=z \text{ and } c_x = i_z\}$ is an $r_i$-separated set in $B(z,e^{-i_z}) \subset Z$, therefore by Ahlfors regularity it has cardinality
	\[
		\preceq \left(\frac{e^{-i_z}}{r_i}\right)^Q
		\preceq \left( \frac{e^{-i_z}}{ (i-s)^{1/Q} e^{-i} } \right)^Q
		= \frac{e^{Q(i-i_z)}}{(i-s)}
		\leq \frac{e^{Q(t-s)}}{t-s}.
		\qedhere
	\]
\end{proof}

\section{Upper bounds for hyperbolic spaces with hyperplanes}\label{section:upperHyp}
In this section we present an approach to finding upper bounds on the $L^p$-Poincar\'e profiles of hyperbolic spaces.  
Our hypotheses are as follows:
\begin{enumerate}
	\item $(X,d,\mu)$ is a $\delta$-hyperbolic geodesic metric measure space, and it is visual with respect to a given point $x_0 \in X$: there exists $C \geq 0$ so that every $x \in X$ belongs to a $C$-quasi-geodesic ray $\gamma:[0,\infty)\ra X$ with $\gamma(0)=x_0$. 
	\item There exists a constant $h(X)>0$ (called the volume entropy) and a constant $C\geq 0$ such that for every $R>0$, $h(X)R-C\leq \log_e(\mu(B(x_0,R)))\leq h(X)R+C$.
	\item There is a visual metric $\rho$ on $\bdry X$ based at $x_0\in X$ with visibility parameter $\eps>0$; i.e., $\rho(\cdot,\cdot)\asymp \exp(-\eps(\cdot|\cdot)_{x_0})$, where $(\cdot|\cdot)_{x_0}$ denotes the Gromov product with respect to $x_0$.
\end{enumerate}
For our last hypothesis, we require the following notion.
\begin{definition}\label{defn:boundaryshadow} Let $(X,d)$ be a metric space and $x_0\in X$. For $C\geq 1$, a subset $A\subseteq X$ is said to be a \textbf{$C$-asymptotic shadow of} $x_0$ if, for every $x\in A$ there is a $C$-quasi-geodesic ray $\gamma_x:[0,\infty)\to X$ with $\gamma_x(0)=x_0$ and $\gamma_x(r_x)=x$ for some $r_x$, and $\gamma_x[r_x,\infty)\subseteq A$. (Recall that a $C$-quasi-geodesic ray is a $(C,C)$-quasi-isometric embedding of $[0,\infty)$.)
\end{definition} 
The final hypothesis only needs to hold for large $a$, where $a\geq 2$ is the constant of thickness in Definition~\ref{mmspPoincare}.
Let $\Isom_\mu(X)$ be the group of $\mu$-preserving isometries of $X$.
\begin{enumerate}[label=(4)]
  \item \label{condition:cutsets} There exist constants $\kappa,N,C>0$ such that for any $a$-thick subspace $Z$ of $X$ with measure at least $N$, there is some $\psi\in \Isom_{\mu}(X)$, and there exist two measurable subsets $H^\pm$ of $X$ which are $C$-asymptotic shadows of $x_0$, and satisfy the inequalities $\rho(\bdry H^+,\bdry H^-)\geq\kappa$, $\mu(\psi(Z)\cap H^+)\geq \kappa \mu(Z)$ and $\mu(\psi(Z)\cap H^-)\geq\kappa\mu(Z)$.
\end{enumerate}

These properties are satisfied for suitable geometric actions of a hyperbolic group, as we will see in \S~\ref{ssec:hypshadows}. 
\begin{proposition}\label{prop:hypgroup-satisfies4}
	If $G$ is a non-elementary hyperbolic group which acts geometrically on a proper geodesic metric measure space $(X,d,\mu)$ and preserves $\mu$, then for any $x_0 \in X$ and visual metric $\rho$ on $\bdry X$ based at $x_0$ with visual parameter $\epsilon$, $(X,d,\mu)$ satisfies properties $(1)$--$(4)$ for suitable $\delta,C$ and $h(X)$.
	Moreover, $(\bdry X, \rho)$ is Ahlfors $h(X)/\epsilon$-regular.
\end{proposition}

Properties $(1)$--$(3)$ are already known to hold in this generality, so our efforts will be focused on property $(4)$. Given these properties, we find the following bounds on the Poincar\'e profile of $X$.  
\begin{theorem}\label{thm:hypubd} 
  Suppose $X$ satisfies conditions (1)--(4) above for some fixed $\delta,C,\eps,\kappa,N$ and set $Q=h(X)/\eps$. Then we have the following bounds on $\Lambda^p_X$:
\[
\Lambda^p_{X,a}(r) \lesssim_{\delta,C,\kappa,N} \left\{
\begin{array}{ll}
r^{\frac{Q-1}{Q}}
&
\textrm{if } p< Q,
\\
r^{\frac{p-1}{p}}\log(r)^{\frac{1}{p}}
&
\textrm{if } p = Q,
\\
r^{\frac{p-1}{p}}
&
\textrm{if } p > Q.
\end{array}\right.
\]
\end{theorem}

To find the best bound possible for the Poincar\'e profiles $\Lambda_G^p$ of a hyperbolic group $G$ from the above theorem, it is natural to consider the following concept.
\begin{definition}\label{def:equiv-confdim} The \textbf{equivariant conformal dimension} of a hyperbolic group $G$ is defined to be the infimum of the Hausdorff dimension of $(\bdry X,\rho)$ where $\bdry X$ is the boundary of a space $X$ on which $G$ acts geometrically and $\rho$ is a visual metric on $\bdry X$. We say the equivariant conformal dimension is attained if the infimum is realised. 
\end{definition}
An equivalent definition is to minimise $h(X)/\epsilon$ over all such actions, metrics and permissible visibility parameters, thus optimising the bounds from Theorem~\ref{thm:hypubd} (\cite{Coo-93-meas-bdry}, cf.\ proof of Proposition~\ref{prop:hypgroup-satisfies4}).

We note that this quantity can be compared to Pansu's conformal dimension~\cite{Pan-89-cdim}, an important invariant in the study of boundaries of hyperbolic spaces and analysis on metric spaces; we state the Ahlfors regular variation of this definition as if $G$ acts on $X$ geometrically, visual metrics on $\bdry X$ are Ahlfors regular.  (For discussion see e.g.~\cite{Mac-Tys-cdimexpo}.)
\begin{definition}\label{def:confdim}
	The \textbf{(Ahlfors regular) conformal dimension} of a hyperbolic group $G$ is the infimum of the Hausdorff dimensions of $(\bdry G, \rho)$ where $\rho$ is a metric on $\bdry G$ that is Ahlfors regular and also quasisymmetric to some visual metric.
\end{definition}
We do not define ``quasisymmetric'' maps here, but note that by work of Bonk--Schramm~\cite{BS-00-gro-hyp-embed}, a metric on $\bdry G$ is quasisymmetric to a visual metric if and only if it is a visual metric on a space $X$ quasi-isometric to $G$.  Therefore, the conformal dimension is bounded above by the equivariant conformal dimension of Definition~\ref{def:equiv-confdim}.  
Conjecturally the two quantities are equal (Conjecture of Kleiner \cite[Problem 61]{Kap05-problem-list-boundaries}). 

Using Proposition~\ref{prop:hypgroup-satisfies4} and Theorem~\ref{thm:hypubd} we are now ready to prove Theorem~\ref{bthm:hypupperbd}.

\begin{corollary}\label{cor:eqConfdim} Let $G$ be a hyperbolic group and let $Q$ be its equivariant conformal dimension. Then, for any $\epsilon>0$,
\[
 \Lambda^p_G(r) \lesssim \left\{
 \begin{array}{lcl}
 r^{\frac{Q-1}{Q}+\epsilon} & \textup{if} & p \leq Q
 \\
 r^{\frac{p-1}{p}} & \textup{if} & p>Q.
 \end{array}\right.
\]
If the equivariant conformal dimension is attained, we have:
\[
 \Lambda^p_G(r) \lesssim \left\{
 \begin{array}{lcl}
 r^{\frac{Q-1}{Q}} & \textup{if} & 1\leq p < Q
 \\
 r^{\frac{Q-1}{Q}} \log^{\frac{1}{Q}}(r) & \textup{if} & p = Q
 \\
 r^{\frac{p-1}{p}} & \textup{if} & p>Q.
 \end{array}\right.
\]
\end{corollary}

\subsection{Helly's theorem and centrepoints}

Inspired by the arguments presented in \cite[Section 4]{BenSchTim-12-separation-graphs}, we show that finite measure thick subsets of real hyperbolic spaces have ``medians''.
To find a suitable centrepoint of a subset, we use Helly's theorem (cf.\ \cite{Mil-Ten-Thur-Vav-97-Separators}).
The version suitable for our needs is the following variation on a result of Ivanov~\cite{Ivanov_Helly}.
\begin{theorem}[Ivanov]\label{thm:helly}
	Let $X$ be a uniquely geodesic space of compact topological dimension $k < \infty$.
	Let $\cH$ be a (possibly infinite) collection of closed convex subsets of $X$,
	with the property that there exists a compact convex set $Y \subset X$ so that
	for any $H_1,\ldots,H_{k+1} \in \cH$ we have $Y \cap H_1 \cap \cdots \cap H_{k+1} \neq \emptyset$.
	Then $\bigcap_{H \in \cH} H \supset \bigcap_{H \in \cH} H \cap Y \neq \emptyset$.
\end{theorem}
The `compact topological dimension' of a space $X$ is the maximum topological dimension of all compact subsets of $X$.
\begin{proof}
	If not, then for any $y \in Y$ there exists $H_y \in \cH$ with $y \notin H_y$.
	Since $Y$ is compact, for some $y_1,\ldots, y_m$ we have that $\{X\setminus H_{y_i}\}_{i=1,\ldots,m}$ is a finite subcover of the open cover $\{X\setminus H_y\}_{y \in Y}$ of $Y$.
	By assumption, any $k+1$ of the finite collection of convex sets $\{Y,H_{y_1},\ldots,H_{y_m}\}$ have non-empty intersection (in $Y$), and so Helly's Theorem \cite[Theorem~1.1]{Ivanov_Helly} implies that there exists $y \in Y \cap H_{y_1}\cap\cdots\cap H_{y_m} \neq \emptyset$.
	This is a contradiction, since $y$ is not covered by $\{X\setminus H_{y_i}\}_{i=1,\ldots,m}$.
\end{proof}

\begin{lemma}\label{realhypcentrepoint}(Centrepoint theorem) Let $a>0$. There exists a constant $c=c(k,a)>0$ such that for any $a$-thick subset $Z$ of $\HH_\R^k$ with finite measure, there is a point $x\in \HH_\R^k$ such that for any half-space $H$ of $\HH_\R^k$ containing $x$, we have $\mu(H\cap Z)\geq c\mu(Z)$.
\end{lemma}
\begin{proof} By assumption $Z=\bigcup_{i\in I} B(z_i,a)$ for some $\{z_i\}_{i\in I} \subset Z$.  Let $Z'$ be an $2a$-separated $4a$-net in $\setcon{z_i}{i\in I}$. It follows that $\abs{Z'}\asymp_a\mu(Z)$ since $\abs{Z'}\mu(B(z_i,a))\leq \mu(Z) \leq \abs{Z'}\mu(B(z_i,5a))$ for some (any) $z_i$.

Let $Y$ be a large closed (convex) ball containing $Z'$.
Let $\mathcal{Z}$ be the set of all closed half-spaces of $\HH_\R^k$ containing more than $\frac{k}{k+1}\abs{Z'}$ of the points in $Z'$.  Thus the intersection of any $k+1$ of the sets in $\mathcal{Z}$ has non-empty intersection with $Y$.

Applying Theorem~\ref{thm:helly}, and the fact that $\HH_\R^k$ has compact topological dimension $k$, there exists some $x\in \bigcap_{H \in \mathcal{Z}} H$.  Thus for any half-space $H\subset \HH_\R^k$ with $\abs{H\cap Z'}>\frac{k}{k+1}\abs{Z'}$ we have $x\in H$. 
It is a short exercise to see that $x$ is contained in every half-space $H$ such that $\abs{Z'\cap H}>\frac{k}{k+1}\abs{Z'}$ if and only if every half-space $H$ containing $x$ satisfies $\abs{Z'\cap H}>\frac{1}{k+1}\abs{Z'}$.

Let $H$ be a half-space containing $x$ and let $Z'_H=Z'\cap H$. It is clear that $\mu(B(z,r)\cap H)\geq \frac12\mu(B(z,r))$ for any $z\in Z'_H$ and any $r\geq 0$, so
\[
	\mu(Z\cap H) \geq \frac{\mu(B(z,a))}{2(k+1)}\abs{Z'} \asymp_{k,a} \mu(Z).\qedhere
\]
\end{proof}
We can use a measure-preserving isometry to move such a centrepoint $x$ to the origin $o \in \HH_\R^k$ in the Poincar\'e ball model, and now show that hypothesis \ref{condition:cutsets} of Theorem~\ref{thm:hypubd} is satisfied for $\HH_\R^k$.
\begin{lemma}\label{lem:shadows-real-hyp}
  There exist constants $\kappa,C>0$ so that for any $a$-thick subset $Z \subset \HH_\R^k$, and $o\in \HH_\R^k$ a centrepoint of $Z$, there exist $C$-asymptotic shadows of $o$ denoted by $H^-,H^+ \subset \HH_\R^k$ so that we have $\rho(\bdry H^-, \bdry H^+) \geq \kappa$ and that $\mu(Z \cap H^-), \mu(Z \cap H^+) \geq \kappa\mu(Z)$.
\end{lemma}
\begin{proof}
  Fix $a>0$ and $c=c(k,a)>0$ the constants from Lemma~\ref{realhypcentrepoint}.

  Let $H\subset \mathbb{H}^k_\R$ be a hyperplane containing $o$, and let $\alpha>0$.  We denote by $H^{\alpha}$ the union of all two-sided geodesics passing through $o$ and with end points in the $\alpha$-neighbourhood of the boundary $\bdry H\subset \bdry \mathbb{H}^k_\R = \Sph^{k-1}$. 

 We start with an argument inspired by the proof of \cite[Proposition 4.1]{BenSchTim-12-separation-graphs}.  Consider for every $r>0$ the sphere $S_r=\{x\in \mathbb{H}^k_\R, \; d(x,o)=r\}$ equipped with its Riemannian measure $\nu_r$. Note that  
$$\nu_r(S_r\cap H^{\alpha})= \eta(\alpha) \nu_r(S_r)$$ for some increasing function $\eta$ satisfying $\lim_{\alpha\to 0}\eta(\alpha)=0$.  We now fix $\alpha>0$ so that $\eta(\alpha)\leq \frac{c}{2}$.

Recall that hyperplanes passing through $o$ are characterized by their normal vector at $o$, and therefore are parametrized by the projective space $P^{k-1}$. We consider the Lebesgue probability measure $\nu$ on $P^{k-1}$. Given $\theta\in P^{k-1}$ we define $H_\theta$ to be the hyperplane through $o$ with normal vector $\theta$.
Recall that $Z \subset \HH_\R^k$ is a measurable subset of finite measure, so for each $r$
$$\int_{P^{k-1}} \nu_r(Z\cap H_{\theta}^{\alpha}\cap S_r)d\nu(\theta)=\nu_r(Z\cap S_r)\frac{\nu_r(S_r\cap H^{\alpha})}{\nu_r(S_r)}= \nu_r(Z\cap S_r) \eta(\alpha) .$$
Integrating over $r$, we deduce that 
\begin{equation*}
  \int_{P^{k-1}} \mu(Z\cap H_{\theta}^{\alpha})d\nu(\theta)=\mu(Z)\eta(\alpha),
\end{equation*}
and so for some hyperplane $H_Z$ we have $\mu(Z \cap H_Z^\alpha) \leq \mu(Z)\eta(\alpha) \leq \frac{c}{2}\mu(Z)$.

Let $H^-,H^+$ be the two connected components of the complement of $H_Z^\alpha$; these are convex and asymptotic shadows of $o$, and satisfy $\mu(H^- \cap Z), \mu(H^+ \cap Z) \geq \frac{c}{2}\mu(Z)$.  Moreover, $\rho(\bdry H^-, \bdry H^+) \geq 2\alpha$.
\end{proof}

\subsection{Hyperbolic groups and centrepoints}\label{ssec:hypshadows}
In this subsection, we prove Proposition~\ref{prop:hypgroup-satisfies4}. 
\begin{proof}[Proof of Proposition~\ref{prop:hypgroup-satisfies4}]
  Property (1) follows from a standard argument:  As $G$ is infinite, there are (at least) two distinct points $z_1,z_2\in\bdry X$, and so there is a geodesic line $\gamma$ with endpoints $z_1$ and $z_2$ (by e.g.\ \cite[Proposition~7.6]{Ghys-dlH-90-hyp-groups}).  Given $x \in X$, as the action of $G$ on $X$ is cocompact, there exists $g\in G$ so that $g\cdot\gamma$ is within bounded distance of $x$.  As the geodesic triangle between $x_0 \in X$ and $g z_1, g z_2 \in \bdry X$ is $2\delta$-thin,  $x$ must be within a bounded distance of either the geodesic ray from $x_0$ to $gz_1$, or the geodesic ray from $x_0$ to $gz_2$; this ray can be adjusted to a $C$-quasi-geodesic ray that passes through $x$ for some uniform $C$.

Property (2) follows from~\cite[Theorem 7.2]{Coo-93-meas-bdry}, and $(\bdry X, \rho)$ is Ahlfors $Q$-regular with $Q = \frac{1}{\epsilon} h(X)$.  Property (3) is the definition of a visual metric, so it remains only to show that property \ref{condition:cutsets} is satisfied.

%
%
%


We require a probably well-known basic fact about convex hulls of quasi-convex subsets of real hyperbolic spaces. Recall that a subset $Y$ of a geodesic metric space is $K$-quasi-convex if every geodesic that connect a pair of points of $Y$ lies within the $K$-neighbourhood of $Y$. It turns out that in real hyperbolic spaces, quasi-convex subsets are ``nearly'' convex in a stronger sense:

\begin{lemma}\label{lem:Convex/quasiconvex}
Given $K\geq 0$, there exists $N=N(K,k)$ such that for every $K$-quasi-convex subset $Z\subset \HH^k_\R$, the convex hull of $Z$ is contained in the $N$-neighbourhood of $Z$. 
\end{lemma}
\begin{proof}
Note that in Klein model of $\HH^k_\R$, the hyperbolic convex hull coincides with the Euclidean one. By Carath\'eodory's theorem, we deduce that any point of the convex hull of $Z$ is a convex combination of some points $z_1,\ldots, z_m\in Z$, with $m\leq k+1$. Using the quasi-convexity of $Z$, the lemma follows by induction on $m$.
\end{proof}
%
 
We now show that \ref{condition:cutsets} holds for $X$.
Let $X$ be a $\delta_X$-hyperbolic Cayley graph of the hyperbolic group $G$.
By a result of Bonk--Schramm~\cite{BS-00-gro-hyp-embed}, there exists constants $k \in \N, \lambda_\psi \geq 1, C_\psi \geq 0$ and a $(\lambda_\psi,C_\psi)$-quasi-isometric embedding $\psi: X \ra \HH^k_\R$.
By post-composing $\psi$ with an appropriate element of $\Isom_\mu(\HH^k_\R)$ if necessary, we may assume $\psi(1)=o$, the origin in the Poincar\'{e} ball model of $\HH^k_\R$.

  Given a finite subset $Y$ of $VX$, define $Y'\subset \HH^k_\R$ to be the closed $2$-neighbourhood of $\psi(Y)$. By Lemma~\ref{realhypcentrepoint}, there is a constant $c=c(k) > 0$ and a point $x'\in \HH^k_\R$ such that for any half-space $H$ of $\HH^k_\R$ containing $X$ we have $\mu(H\cap Y')\geq c\mu(Y')$. 
  Such $x'$ is contained in the convex hull of $\psi(Y)$, so by Lemma~\ref{lem:Convex/quasiconvex}, $d_{\HH^k_\R}(x',\psi(x))\leq N(k)$ for some $x\in X$. By applying a left-translation in $G$ (by an element $g$) we may assume $x=1$, while by applying an isometry $\phi\in\Isom(\HH^k_\R)$, we may assume $x'=o$. Define $f=\phi\circ \psi \circ g^{-1}:X\to \HH^k_\R$ and let $\bdry f$ be the induced map $\bdry f:\bdry X\to \Sph^{k-1}$, where $\bdry X$ is endowed with a visual metric $\rho$ based at $1$ and $\Sph^{k-1} = \bdry \HH_\R^k$ is endowed with the Euclidean (visual) metric $\rho_{Euc}$.

  By Lemma~\ref{lem:shadows-real-hyp}, there exist constants $\kappa,C$ and $C$-asymptotic shadows of $o$ denoted $H^\pm$ so that $\rho_{Euc}(\bdry H^-,\bdry H^+)\geq 4\kappa$ and that $\mu(Y'\cap H^\pm)\geq 4\kappa\mu(Y')$.  Since $f(1)=o$, it follows that 
  \[ \rho(\bdry f^{-1}[\bdry H^-]_{\kappa},\bdry f^{-1}[\bdry H^+]_\kappa)\geq \kappa' \]
  for some $\kappa'>0$ which does not depend on the choices of $\phi$ and $g$ used to construct $f$.  (It is not \textit{a priori} obvious that either of $\bdry f^{-1}[\bdry H^\pm]_{\kappa}$ is non-empty.)

  Define $H^\pm_X$ to be the set of all points $y \in X \setminus B(1,R)$ contained in the $A$-neighbourhood of the set of all geodesic rays in $X$ from $1$ to a point in $\bdry f^{-1}([\bdry H^\pm]_{\kappa})$, where $A$ and $R$ are determined below.
  
  
  We claim that there exist $A, R$ so that if $y\in Y$ satisfies $d_X(1,y) \geq R$ and $B(f(y),2)\cap H^\pm\neq\emptyset$, then $y\in H^\pm_X$. 
Let $z\in H^\pm$ satisfy $d_{\HH^k_\R}(z,f(y))\leq 2$, and let $\gamma$ be the unique geodesic ray in $\HH^k_\R$ starting at $o$ and containing $z$ (we assume $y$ is sufficiently far from $1$ that $z\neq o$); denote the boundary point of $\gamma$ by $\zeta$. 
Since $X$ is $C_X$-visual for some $C_X$, there exists a $C_X$-quasi-geodesic ray $\beta$ in $X$ from $1$ that contains $y$; denote by $\eta$ the boundary point of $\beta$ in $\bdry X$.
The Gromov product of $\zeta$ and $\bdry f(\eta)$ (relative to $o$) is bounded from below by $d_{\HH^k_\R}(o,f(y))$ up to a uniform additive error, so by insisting that $d_X(1,y)\geq R$ is sufficiently large, we may assume that $\rho_{Euc}(\zeta,\bdry f(\eta))\leq \kappa$, hence $\bdry f(\eta)\in [\bdry H^\pm]_{\kappa}$.
By the Morse Lemma, $\beta$ is contained in a uniform neighbourhood of a geodesic ray from $1$ to $\eta$, and hence for a suitable choice of $A$ will be contained in $H^\pm_X$ outside $B(1,R)$.  For these choices of $R,A$ we have that $y \in H^\pm_X$ as desired.

  From this, and the fact that $f$ is a quasi-isometry with fixed constants, it follows that there exist $\eta, \eta' >0$ so that
$|Y \cap H_X^\pm| \geq \eta \mu(Y' \cap H^\pm) \geq \eta \kappa \mu(Y') \geq \eta \kappa \eta' |Y|$.

The proof of Proposition~\ref{prop:hypgroup-satisfies4} is complete.
\end{proof}

\subsection{Upper bounds for the Poincar\'e profile}
\begin{proof}[Proof of Theorem~\ref{thm:hypubd}] 
  Let $x_0\in X$ and $a\geq 2$ be fixed so that \ref{condition:cutsets} holds.  Let $Z$ be an $a$-thick subspace of $X$ of sufficiently large finite measure (to be determined later).  Apply \ref{condition:cutsets} to move $Z$; without loss of generality we may assume that $\psi=id$.  Let $H^\pm$ be the corresponding $C$-asymptotic shadows of $x_0$.
  
Define $\bdry\phi:(\bdry X,\rho)\to [0,1]$ by
\[
\bdry\phi(z) = \min\{1, \max\{0, \tfrac{3}{\kappa}\rho(z,\bdry H^-)-1\}\};
\]
this is a $\frac{3}{\kappa}$-Lipschitz function so that $\bdry\phi$ is zero on $[\bdry H^-]_{\kappa/3}$ and one on $[\bdry H^+]_{\kappa/3}$.

We choose a function $\phi:X\to[0,1]$ by setting $\phi(x) = \bdry\phi(\eta)$ where $\eta\in \bdry X$ is the endpoint of some $C$-quasi-geodesic $\gamma_x:[0,\infty)\ra X$ with $\gamma_x(0)=x_0$ and $\gamma_x(t)=x$ for some $t$.
Regardless of the choices made in defining this function we have the following control: for any $x,y \in X$ with $d(x,y) \leq C'$ there exists $K=K(\delta, C, C', \rho, \kappa)$ so that
\begin{equation}\label{eq:lip-ext}
 \abs{\phi(x)-\phi(y)}\leq K\exp(-\eps d(x,x_0)).
\end{equation}
By a similar argument, there exists $L >0$ so that if $d(x,x_0) \geq L$ and $x \in H^-$ then the endpoint $\eta$ of $\gamma_x$ used to define $\phi(x)$ satisfies $\rho(\eta, \bdry H^-) \leq \kappa/3$, and so $\phi(x)=0$.  Likewise, if $x \in H^+$ and $d(x,x_0) \geq L$ then $\phi(x)=1$. 

Note that $\phi$ might not be measurable, but it is easy to see how it can be slightly modified so that it, and its gradient, are measurable: consider a measurable partition of the space, whose subsets have diameter at most $1$, and on each subset replace $\phi$ by its maximum over this subset.  

By assuming that $\mu(Z)$ is greater than $\frac{2}{\kappa} \mu(B(x_0,L))$, we know---by assumption \ref{condition:cutsets}---that $\mu(Z \cap H^- \setminus B(x_0,L))$ and $\mu(Z \cap H^+ \setminus B(x_0,L))$ are both $\geq \frac{\kappa}{2}\mu(Z)$.  Switching the roles of $H^\pm$ if necessary, we assume $\phi_Z \geq 1/2$ and so
\begin{equation}\label{eq:hyp-norm-bound}
  \norm{\phi-\phi_Z}_{Z,p}^p \geq |\phi_Z|^p \mu(Z \cap H^- \setminus B(x_0,L)) \geq 2^{-p-1} \kappa \mu(Z).
\end{equation}

We now bound $\norm{\nabla_a \phi}_{B,p}$ on the ball $B=B(x_0,r)$. Since we have $\mu(B(x_0,R)) \asymp \exp(h(X)R)$, \eqref{eq:lip-ext} gives 
\begin{equation}\label{eq:lip-ball-bound}
  \norm{\nabla_a \phi}_{B,p}^p
  \preceq_{K,\kappa,p}
  \int_{t=0}^r \exp(h(X)t)\exp(-p\eps t) dt.
\end{equation}

We now consider the three cases for $p$ separately.
(Recall that $h(X)=\epsilon Q$.)

\smallskip
\textbf{Case $1$, $p>Q$:}
Equation \eqref{eq:lip-ball-bound} gives that $\norm{\nabla_a \phi}_{X,p}^p$ is bounded by some constant $D$ only depending on $K$, $\kappa$ and $p$, so \eqref{eq:hyp-norm-bound} gives $h_a^p(Z) \preceq_{K,\kappa,p} \mu(Z)^{-1/p}$ for any subspace $Z$ and the case $p>Q$ follows.  

\smallskip
\textbf{Case 2, $p<Q$:}
By \eqref{eq:lip-ext}, we have
$ |\nabla_a \phi|(x) \preceq \exp(-\eps d(x,x_0)) $, so
\begin{equation*}
	\| \nabla_a \phi \|_{Z,p}^p \preceq \| \exp(-\eps d(\cdot,x_0)) \|_{Z,p}^p.
\end{equation*}
Since $\exp(-\eps d(\cdot,x_0))$ is a decreasing function of the distance to $x_0$, for every $r$ such that $\mu(B(x_0,r)) \geq \mu(Z)$ we have
\begin{align*}
   & \|\exp(-\eps d(\cdot,x_0)) \|_{Z,p}^p
   \\ & = 
   \|\exp(-\eps d(\cdot,x_0)) \|_{Z\cap B(x_0,r),p}^p
	+  \|\exp(-\eps d(\cdot,x_0)) \|_{Z\setminus B(x_0,r),p}^p
   \\ & \leq \|\exp(-\eps d(\cdot,x_0)) \|_{Z \cap B(x_0,r),p}^p
	+   \|\exp(-\eps d(\cdot,x_0)) \|_{B(x_0,r)\setminus Z,p}^p
   \\ & = \|\exp(-\eps d(\cdot,x_0)) \|_{B(x_0,r),p}^p.
\end{align*}
Pick $r$ such that $\mu(B(x_0,r-1))\leq \mu(Z)\leq \mu(B(x_0,r))$ so that $\mu(Z)\asymp \mu(B(x_0,r))$.
Therefore, as in \eqref{eq:lip-ball-bound}, we have 
\begin{align*}
  \norm{\nabla_a \phi}_{B(x_0,r),p}^p 
  & \preceq 
  \|\exp(-\epsilon d(\cdot,x_0))\|_{B(x_0,r)}^p
  \\ & \preceq \int_{t=0}^r \exp(h(X)t)\exp(-p\epsilon t) dt
  \\ & \preceq
  \exp(h(X)r) \cdot \exp(-p\eps r) 
  \asymp \mu(Z) \cdot \mu(Z)^{-p/Q},
\end{align*}
thus $h_a^p(Z) \preceq \mu(Z)^{-1/Q}$ and the bound on $\Lambda^p_{X,a}(\mu(Z))$ follows.

\smallskip
\textbf{Case 3, $p=Q$:}
If $p=Q$ then the same argument as in Case 2 shows that $\|\exp(-\epsilon d(\cdot,x_0))\|_{Z,p}^p$ is maximised for a metric ball, so
\[
\norm{\nabla_a \phi}_{Z,p}^p \preceq \int_{t=0}^r \exp(h(X)t)\exp(-p\eps t) dt = r \asymp \log(\mu(Z)),
\]
so $h_a^p(Z) \preceq \log(\mu(Z))^{1/p} \cdot \mu(Z)^{-1/p}$ and thus the bound on $\Lambda^{p}_{X}$ for $p=Q$ follows.
\end{proof}



\section{Applications to buildings and symmetric spaces}\label{sec:applications}
	We use results from Sections~\ref{sec:hyp-PI} and \ref{section:upperHyp} to calculate Poincar\'e profiles of buildings and rank-one symmetric spaces (Theorem~\ref{bthm:hypconfdim}).

Bourdon and Pajot~\cite{BP-99-hyp-build-PI} showed that a family of Fuchsian buildings earlier studied by Bourdon~\cite{Bou-97-GAFA-exact-cdim} have boundaries that admit $1$-Poincar\'e inequalities.  
\begin{definition}
  Let $m \geq 5, n \geq 3$ be given.  Let $R$ be the regular, right-angled hyperbolic polygon with $m$ sides.  Let $I=I_{m,n}$ be the Fuchsian building where the chambers are isometric to $R$, each edge is adjacent to $n$ copies of $R$, and the vertex links are copies of the complete bipartite graph with $n,n$ vertices.

  The group
  \[ G_{m,n}=\langle s_1, \ldots, s_m \,|\, s_i^n, [s_i, s_{i+1}] \, \forall i \rangle, \]
  where indices are modulo $m$,
  acts cellularly and geometrically on $I_{m,n}$.  By \cite[Theorem 1.1]{BP-99-hyp-build-PI}, $\bdry G_{m,n} = \bdry I_{m,n}$ carries an Ahlfors $Q_{m,n}$-regular metric, where
$Q_{m,n} = 1+\log(n-1)/\mathrm{arccosh}((m-2)/m) \in (1, \infty)$,
and  which admits a $1$-Poincar\'e inequality in the sense of Hein\-on\-en--Koskela (Section~\ref{sec:hyp-PI}).
\end{definition}
The apartments in $I_{m,n}$ are each copies of the hyperbolic plane tiled by right-angled regular $m$-gons.  As such, they have separation at least $\log(r)$; the boundary geometry lets us find much larger lower bounds. 
\begin{theorem}
  Given $m\geq 5, n \geq 3$, and $p \in [1, \infty)$,
  \begin{equation*}
    \Lambda_{I_{m,n}}^p(r) \simeq_p
    \begin{cases}
      r^{1-1/{Q_{m,n}}} & \text{ if } p < Q_{m,n} \\
	  r^{1-1/Q_{m,n}} \log(r)^{1/Q_{m,n}} & \text{ if } p=Q_{m,n}\\
      r^{1-1/p} & \text{ if } p > Q_{m,n} .\\
    \end{cases}
  \end{equation*}
\end{theorem}
\begin{proof}
  The lower bounds follow from Theorem~\ref{thm:hyp-PI-bdry} for $p < Q_{m,n}$, Theorem~\ref{thm:hyp-PI-bdry-sharp} for $p=Q_{m,n}$ and Corollary~\ref{cor:hyp-lower-bound} for $p \geq Q_{m,n}$.
  The upper bounds follow from Corollary~\ref{cor:eqConfdim}.
\end{proof}

Finally, we calculate the Poin\-car\'{e} profiles of rank-one symmetric spaces.
The case of $p=1$ for $\HH^k_{\R}$ is dealt with by \cite[Proposition $4.1$]{BenSchTim-12-separation-graphs} and Proposition \ref{prop:Lambda1Sep}, but all other cases are new.

\begin{theorem}\label{pProfhyp} 
	Let $\mathbb{K} \in \{ \R, \C, \HH, \mathbb{O}\}$ be a real division algebra, and let $X=\HH_{\mathbb{K}}^m$ be a rank-one symmetric space for $m \geq 2$ (and $m=2$ when $\mathbb{K}=\mathbb{O}$).
	Let $Q = (m+1)\dim_{\R} \mathbb{K} -2$, then
\[
\Lambda^p_{\mathbb{H}^m_{\mathbb{K}}}(r) \simeq \left\{
\begin{array}{ll}
  r^{\frac{Q-1}{Q}}
&
\textrm{if } p< Q
\\
r^{\frac{Q-1}{Q}}\log(r)^{\frac{1}{Q}}
&
\textrm{if } p = Q
\\
r^{\frac{p-1}{p}}
&
\textrm{if } p > Q
\end{array}\right.
\]
\end{theorem}
\begin{proof}
  The boundary of a rank-one symmetric space carries a visual metric that is Ahlfors $Q$-regular for the given exponent, and satisfies a $1$-Poincar\'e inequality.
  The result then follows from Theorem~\ref{thm:hyp-PI-bdry}, Theorem~\ref{thm:hyp-PI-bdry-sharp}, Corollary~\ref{cor:hyp-lower-bound}, and Theorem~\ref{thm:hypubd}.  
\end{proof}

%
%
%
%

\def\cprime{$'$}

\end{document}